\documentclass[smallextended]{svjour3}
\usepackage{amsmath}
\usepackage{amsfonts}
\usepackage{amssymb}
\usepackage{graphicx}
\usepackage{subfigure}
\usepackage{geometry}
\usepackage[utf8]{inputenc}
\usepackage{float}
\usepackage{tikz}

\newcommand{\ds}{\displaystyle}

\def\pa{\partial}
\def\ukn{U_{K}^{n}}
\def\uknp{U_{K}^{n+1}}

\def\ud{u_{\delta}}

\def\pa{\partial}

\def\xi{x_{i}}

\title{A finite volume scheme for convection-diffusion equations with nonlinear diffusion derived from the Scharfetter-Gummel scheme}
\author{Marianne Bessemoulin-Chatard}
\institute{Marianne Bessemoulin-Chatard \at Université Blaise Pascal - Laboratoire de Mathématiques UMR 6620 - CNRS - Campus des Cézeaux, B.P. 80026 63177 Aubière cedex \\ \email{Marianne.Chatard@math.univ-bpclermont.fr}}

\begin{document}

\titlerunning{A finite volume scheme for convection-diffusion equations with nonlinear diffusion}
\maketitle

\begin{abstract}
We propose a finite volume scheme for convection-diffusion equations with nonlinear diffusion. Such equations arise in numerous physical contexts. We will particularly focus on the drift-diffusion system for semiconductors and the porous media equation. In these two cases, it is shown that the transient solution converges to a steady-state solution as $t$ tends to infinity. \\
The introduced scheme is an extension of the Scharfetter-Gummel scheme for nonlinear diffusion. It remains valid in the degenerate case and preserves steady-states. We prove the convergence of the scheme in the nondegenerate case. Finally, we present some numerical simulations applied to the two physical models introduced and we underline the efficiency of the scheme to preserve long-time behavior of the solutions.
\end{abstract}


\section{Introduction} 

In this article, our aim is to elaborate a finite volume scheme for convection-diffusion equations with nonlinear diffusion. The main objective of building such a scheme is to preserve steady-states in order to be able to apply it to physical models in which it has been proved that the solution converges to equilibrium in long time. In particular, this convergence can be observed in the drift-diffusion system for semiconductors as well as in the porous media equation. \\
In this context, we will first present these two physical models – drift-diffusion system for semiconductors and porous media equation. Then, we will precise the general framework of our study in this article.


\subsection{The drift-diffusion model for semiconductors}

The drift-diffusion system consists of two continuity equations for the electron density $N(x,t)$ and the hole density $P(x,t)$, as well as a Poisson equation for the electrostatic potential $V(x,t)$, for $t \in \mathbb{R}^{+}$ and $x \in \mathbb{R}^{d}$.\\
Let $\Omega \subset \mathbb{R}^{d}$ ($d \geq 1$) be an open and bounded domain. The drift-diffusion system reads 
\begin{equation}
\left\{\begin{array}{lcl} \pa_{t}N-\text{div}(\nabla r(N)-N\nabla V)=0 & \text{ on } & \Omega \times (0,T), \\
\pa_{t}P-\text{div}(\nabla r(P)+P\nabla V)=0 & \text{ on } & \Omega \times (0,T),  \\
\Delta V=N-P-C & \text{ on } & \Omega \times (0,T),\end{array}\right.
\label{DD}
\end{equation}
where $C \in L^{\infty}(\Omega)$ is the prescribed doping profile.\\
The pressure has the form of a power law,
\begin{equation*}
r(s)=s^{\gamma}, \quad \gamma \geq 1.
\end{equation*} 
We supplement these equations with initial conditions $N_{0}(x)$ and $P_{0}(x)$ and physically motivated boundary conditions: the boundary $\Gamma=\pa \Omega$ is split into two parts $\Gamma=\Gamma^{D} \cup \Gamma^{N}$ and the boundary conditions are Dirichlet boundary conditions $ \overline{N}$, $ \overline{P}$ and $ \overline{V}$ on ohmic contacts $\Gamma^{D}$ and homogeneous Neumann boundary conditions on $r(N)$, $r(P)$ and $V$ on insulating boundary segments $\Gamma^{N}$.\\
The large time behavior of the solutions to the nonlinear drift-diffusion model (\ref{DD}) has been studied by A. Jüngel in \cite{Juengel1995a}. It is proved that the solution to the transient system converges to a solution of the thermal equilibrium state as $t \rightarrow \infty$ if the Dirichlet boundary conditions are in thermal equilibrium. The thermal equilibrium is a particular steady-state for which electron and hole currents, namely $\nabla r(N)-N\nabla V$ and $\nabla r(P)+P\nabla V$, vanish. The existence of a thermal equilibrium has been studied in the case of a linear pressure by P. Markowich, C. Ringhofer and C. Schmeiser in \cite{Markowich1986,Markowich1990}, and in the nonlinear case by P. Markowich and A. Unterreiter in \cite{Markowich1993}.\\
We introduce the enthalpy function $h$ defined by
\begin{equation}
h(s)=\int_{1}^{s}\frac{r'(\tau)}{\tau}\, d\tau
\label{h}
\end{equation}
and the generalized inverse $g$ of $h$ defined by
\begin{equation*}
g(s)=\left\{ \begin{array}{rcl} h^{-1}(s) &\text{ if }& h(0^{+})<s<\infty,\\ 0 &\text{ if }& s \leq h(0^{+}).\end{array}\right.
\end{equation*}
If the boundary conditions satisfy $\overline{N},\overline{P}>0$ and
\begin{equation*}
h(\overline{N})-\overline{V}=\alpha_{N} \text{ and } h(\overline{P})+\overline{V}=\alpha_{P} \text{ on } \Gamma^{D},
\end{equation*}
the thermal equilibrium is defined by
\begin{equation}
N^{eq}(x)=g\left(\alpha_{N}+V^{eq}(x)\right), \quad P^{eq}(x)=g\left(\alpha_{P}-V^{eq}(x)\right), \quad x \in \Omega,
\label{eqNP}
\end{equation}
while $V^{eq}$ satisfies the following elliptic problem
\begin{equation}
\left\{\begin{array}{ll} \Delta V^{eq}=g\left(\alpha_{N}+V^{eq}\right)-g\left(\alpha_{P}-V^{eq}\right)-C \text{   in } \Omega, &\\
V^{eq}(x)=\overline{V}(x) \text{ on } \Gamma^{D}, \quad \nabla V^{eq}\cdot \mathbf{n}=0 \text{ on } \Gamma^{N}.&\end{array}\right.
\label{eqV}
\end{equation}
The proof of the convergence to thermal equilibrium is based on an energy estimate with the control of the energy dissipation. More precisely, if we define 
\begin{equation}
H(s)=\int_{1}^{s}h(\tau)d\tau, \quad s \geq 0,
\label{H}
\end{equation}
then we can introduce the deviation of the total energy (sum of the internal energies for the electron and hole densities and the energy due to the electrostatic potential) from the thermal equilibrium (see \cite{Juengel1995a})
\begin{eqnarray}
\mathcal{E}(t)&=&\int_{\Omega}\left(\vphantom{\frac{1}{2}}H\left(N(t)\right)-H\left(N^{eq}\right)-h\left(N^{eq}\right)\left(N(t)-N^{eq}\right)+H\left(P(t)\right)-H\left(P^{eq}\right)\right. \nonumber\\
& & \quad \left. -h\left(P^{eq}\right)\left(P(t)-P^{eq}\right) +\frac{1}{2}\left\vert \nabla \left(V(t)-V^{eq}\right)\right\vert^{2}\right)dx,
\label{Econtinusc}
\end{eqnarray}
and the energy dissipation
\begin{equation}
\mathcal{I}(t)=-\int_{\Omega}\left(N(t)\left\vert \nabla(h(N(t))-V(t))\right\vert^{2}+P(t)\left\vert \nabla(h(P(t))+V(t))\right\vert^{2}\right)dx.
\label{Icontinusc}
\end{equation}
Then the keypoint of the proof is the following estimate:
\begin{equation}
0 \leq \mathcal{E}(t)+\int_{0}^{t}\mathcal{I}(\tau)\,d\tau \leq \mathcal{E}(0).
\label{inegentropie}
\end{equation}


\subsection{The porous media equation}

The flow of a gas in a $d$-dimensional porous medium is classically described by the Leibenzon-Muskat model, 
\begin{equation}
\left\{\begin{array}{lcl} \pa_{t}v=\Delta v^{\gamma} & & \text{ on } \mathbb{R}^{d}\times(0,T),\\ v(x,0)=v_{0}(x) & & \text{ on }  \mathbb{R}^{d},\end{array}\right.
\label{pm}
\end{equation}
where the function $v$ represents the density of the gas in the porous medium and $\gamma>1$ is a physical constant.\\
With a time-dependent scaling (see \cite{Carrillo2000}), we transform (\ref{pm}) into the nonlinear Fokker-Planck equation
\begin{equation}
\left\{\begin{array}{lcl} \pa_{t}u=\text{div}(xu+\nabla u^{\gamma}) & & \text{ on } \mathbb{R}^{d}\times(0,T),\\
 u(x,0)=u_{0}(x)& & \text{ on }  \mathbb{R}^{d}.\end{array}\right.
\label{PM}
\end{equation}
It is proved in \cite{Carrillo2000} that the unique stationary solution of (\ref{PM}) is given by the Barenblatt-Pattle type formula
\begin{equation}
u^{eq}(x)=\left(C_{1}-\frac{\gamma -1}{2\gamma}|x|^{2}\right)_{+}^{1/(\gamma -1)},
\label{barenblatt}
\end{equation}
where $C_{1}$ is a constant such that $u^{eq}$ has the same mass as the initial data $u_{0}$.\\ 
Moreover, J. A. Carrillo and G. Toscani have proved in \cite{Carrillo2000} the convergence of the solution $u(x,t)$ of (\ref{pm}) to the Barenblatt-Pattle solution $u^{eq}(x)$ as $t \rightarrow \infty$. As in the case of the drift-diffusion model, the proof of the convergence to the Barenblatt-Pattle solution is based on an entropy estimate with the control of the entropy dissipation given by (\ref{inegentropie}), where the relative entropy is defined by
\begin{equation}
\mathcal{E}(t)=\int_{\mathbb{R}^{d}}\left(H(u(t))-H(u^{eq})+\frac{|x|^{2}}{2}\left(u(t)-u^{eq}\right)\right)dx,
\label{Econtinupm}
\end{equation}
where $H$ is defined by (\ref{H}) and the entropy dissipation is given by
\begin{equation}
\mathcal{I}(t)=-\frac{d}{dt}\mathcal{E}(t)=-\int_{\mathbb{R}^{d}}u(t)\left\vert \nabla\left(h(u(t))+\frac{|x|^{2}}{2}\right) \right\vert^{2}dx.
\label{Icontinupm}
\end{equation}


\subsection{Motivation}

Many numerical schemes have been proposed to approximate the solutions of nonlinear con\-vection-diffusion equations. In particular, finite volume methods have been proved to be efficient in the case of degenerate parabolic equations (see \cite{Eymard2000,Eymard2002}). We also mention the combined finite volume-finite element approach for nonlinear degenerate parabolic convection-diffusion-reaction equations analysed in \cite{Eymard2006a}. The definition of the so-called local Péclet upstream weighting numerical flux guarantees the stability of the scheme while reducing the excessive numerical diffusion added by the classical upwinding.\\
On the other hand, there exists a wide literature on numerical schemes for the drift-diffusion equations. It started with 1-D finite difference methods and the Scharfetter-Gummel scheme (\cite{Scharfetter1969}). In the linear pressure case ($r(s)=s$), a mixed exponential fitting finite element scheme has been successfully developed by F. Brezzi, L. Marini and P. Pietra in \cite{Brezzi1987,Brezzi1989}. The adaptation of the mixed exponential fitting method to the nonlinear case has been developed by F. Arimburgo, C. Baiocchi, L. Marini in \cite{Arimburgo1992} and by A. Jüngel in \cite{Juengel1995} for the one-dimensional problem, and by A. Jüngel and P. Pietra in \cite{Juengel1997} for the two-dimensional problem. Moreover, C. Chainais-Hillairet and Y.J. Peng proposed a finite volume scheme for the drift-diffusion equations in 1-D in \cite{Chainais-Hillairet2003a}, which was extended in \cite{Chainais-Hillairet2003,Chainais-Hillairet2004} in the multidimensional case. C. Chainais-Hillairet and F. Filbet also introduced in \cite{Chainais-Hillairet2007} a finite-volume scheme preserving the large time behavior of the solutions of the nonlinear drift-diffusion model.\\ 
Now to explain our approach, let us first recall some previous numerical results concerning the drift-diffusion system for semiconductors. The precise definitions of schemes considered will be presented in Section 2. We compare results obtained with three existing finite volume schemes: the classical upwind scheme proposed by C. Chainais-Hillairet and Y. J. Peng in \cite{Chainais-Hillairet2003a}, the Scharfetter-Gummel scheme introduced in \cite{Scharfetter1969} and the nonlinear upwind scheme studied in \cite{Chainais-Hillairet2007}.\\
In Figure \ref{comparelin}, we present some results obtained in the case of a linear diffusion ($r(s)=s$). We represent the relative energy $\mathcal{E}$ and the dissipation of energy $\mathcal{I}$ obtained with the upwind flux and the Scharfetter-Gummel flux for a test case in one space dimension. We can observe a phenomenon of saturation of $\mathcal{E}$ and $\mathcal{I}$ for the upwind flux. In addition, we clearly observe that the energy and its dissipation obtained with the Scharfetter-Gummel flux converge to zero when time goes to infinity, which means that densities $N(t)$ and $P(t)$ converge to the thermal equilibrium. It appears that the Scharfetter-Gummel flux is very efficient, but is only valid for linear diffusion. Moreover, we can emphasize that contrary to the upwind flux, the Scharfetter-Gummel flux preserves the thermal equilibrium.\\
In Figure \ref{comparenonlin}, we present numerical results obtained in the case of a nonlinear diffusion $r(s)=s^{2}$. We represent the relative energy $\mathcal{E}$ and the dissipation $\mathcal{I}$ obtained with the classical upwind flux and with the nonlinear upwind flux for a test case in one dimension of space. We still observe a phenomenon of saturation of $\mathcal{E}$ and $\mathcal{I}$ for the classical upwind flux. For the nonlinear flux, we clearly notice that the energy and its dissipation converge to zero when time goes to infinity.\\
Looking at these results, it seems crucial that the numerical flux preserves the thermal equilibrium to obtain the consistency of the approximate solution in the long time asymptotic limit.

\begin{figure}
\centering
\subfigure{\includegraphics[scale=0.45]{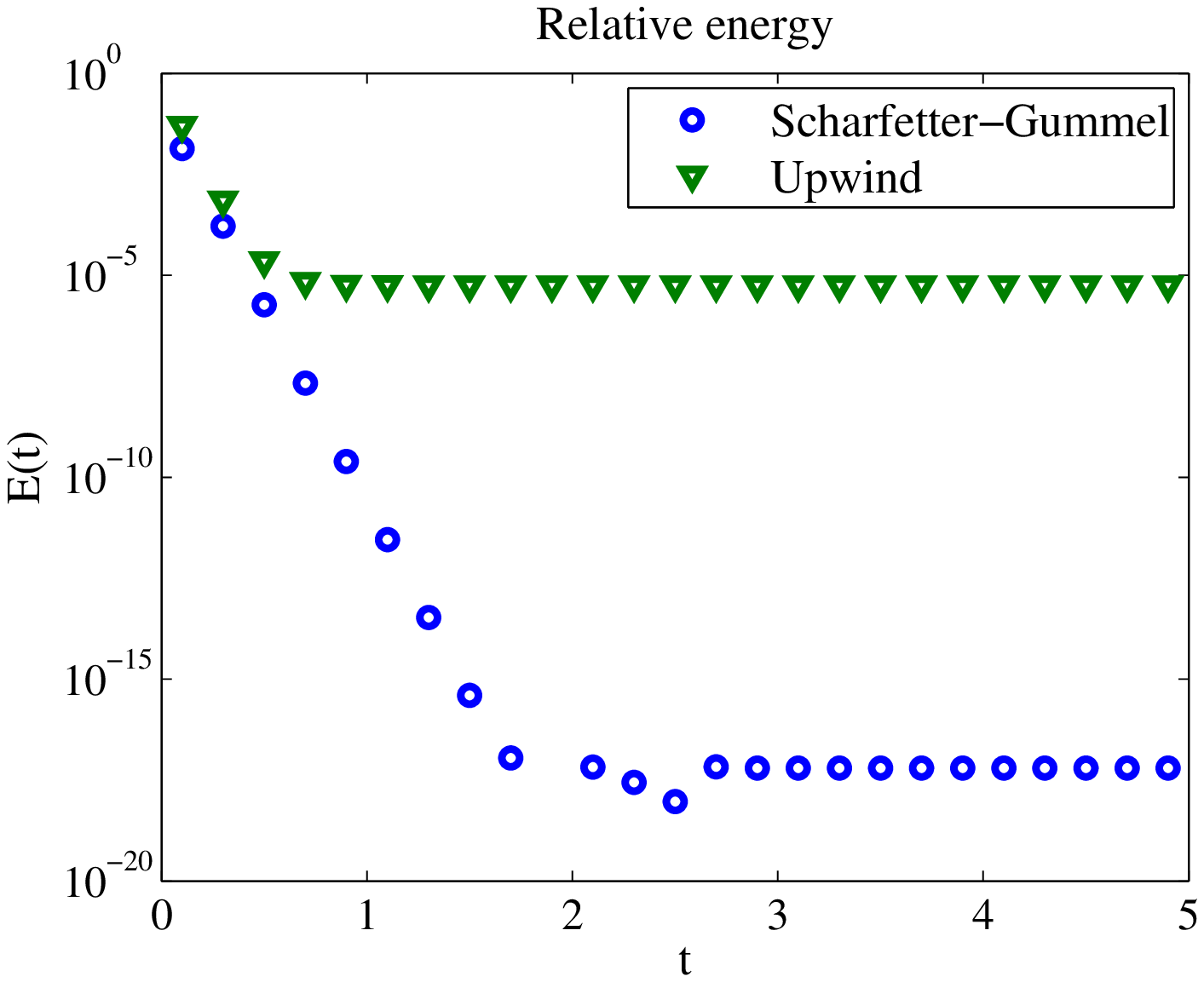}}
\subfigure{\includegraphics[scale=0.45]{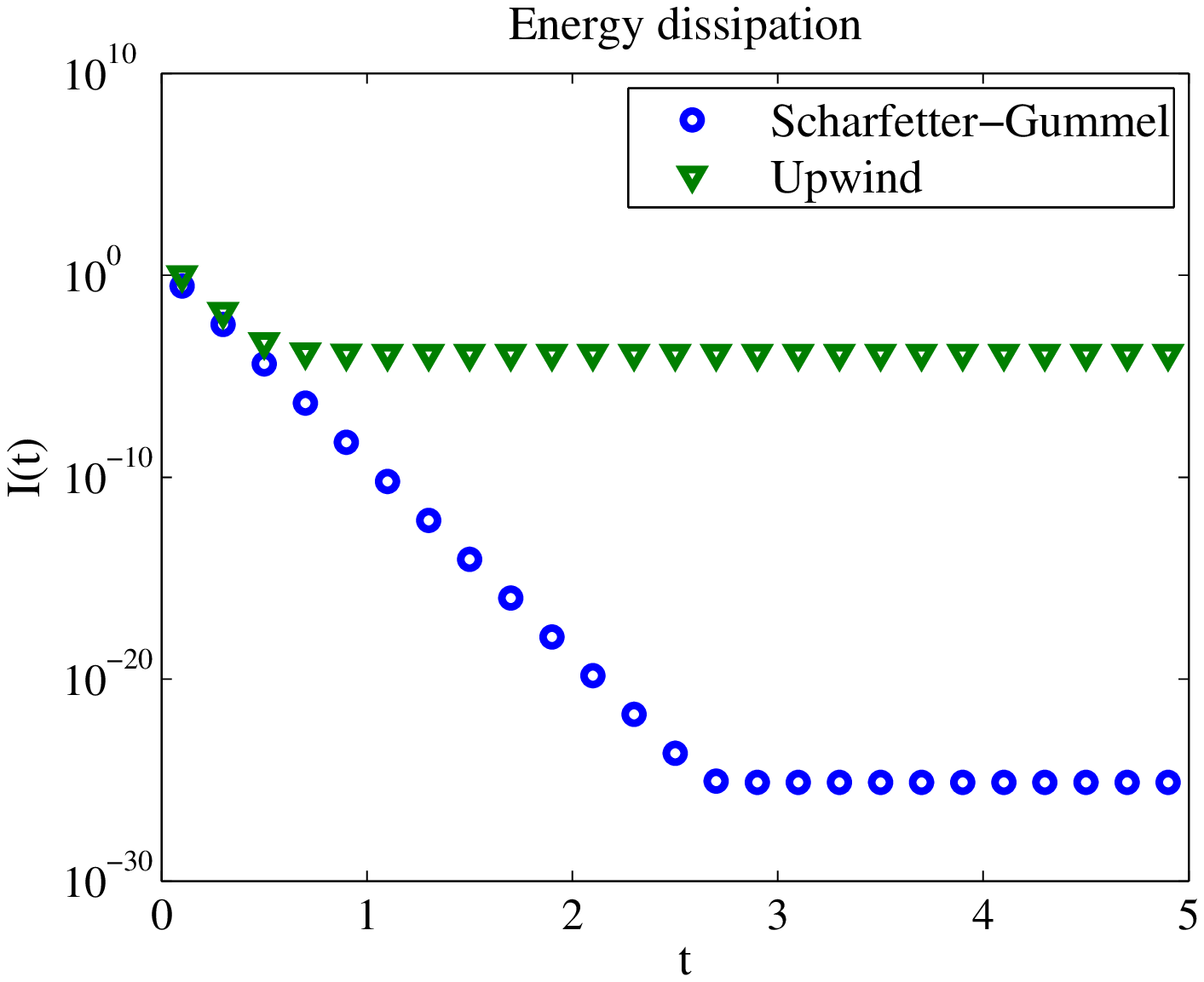}}
\caption{Linear case: relative energy $\mathcal{E}^{n}$ and dissipation $\mathcal{I}^{n}$ for different schemes in log scale, with time step $\Delta t=10^{-2}$ and space step $\Delta x=10^{-2}$.}
\label{comparelin}
\end{figure}

\begin{figure}
\centering
\subfigure{\includegraphics[scale=0.45]{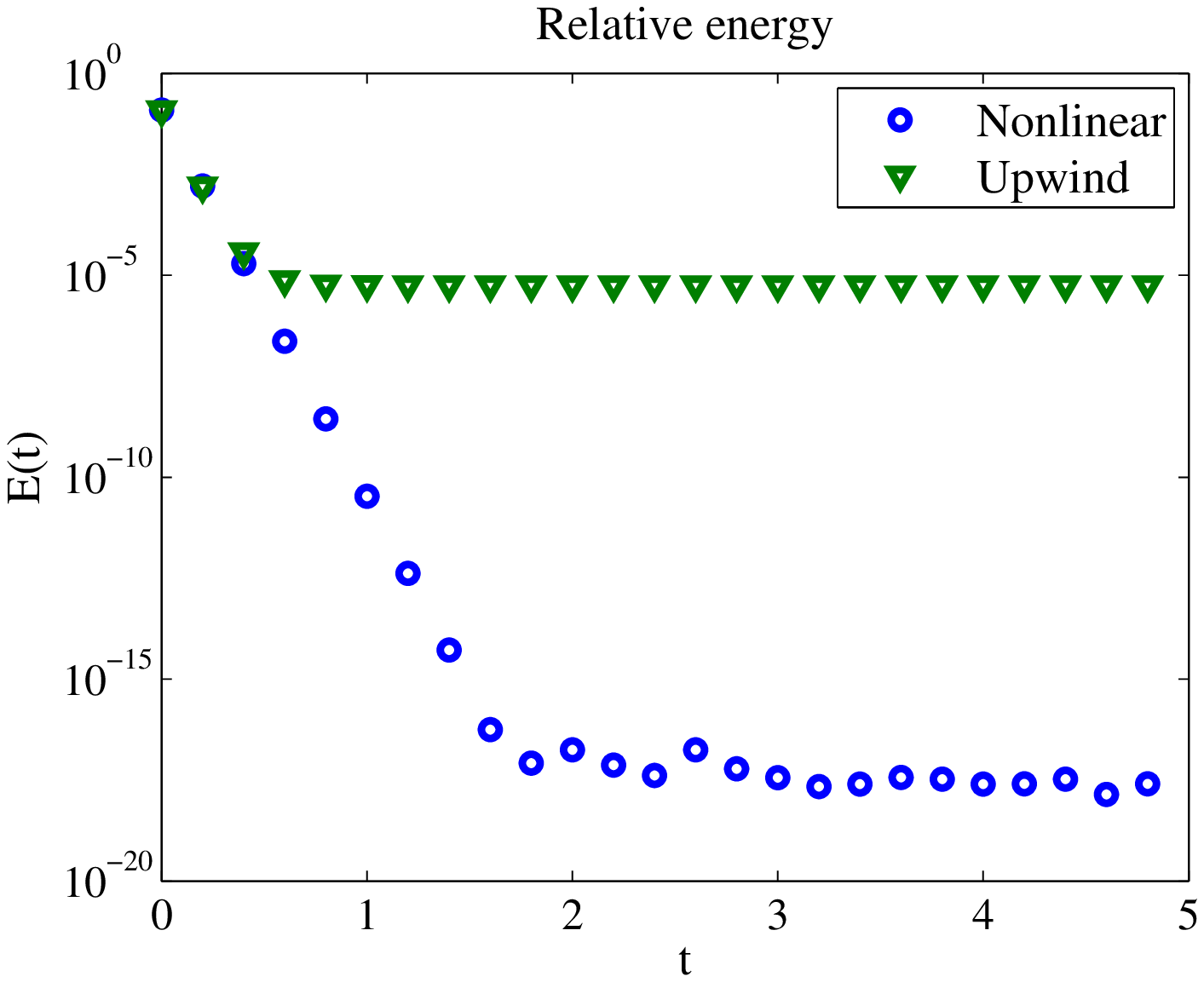}}
\subfigure{\includegraphics[scale=0.45]{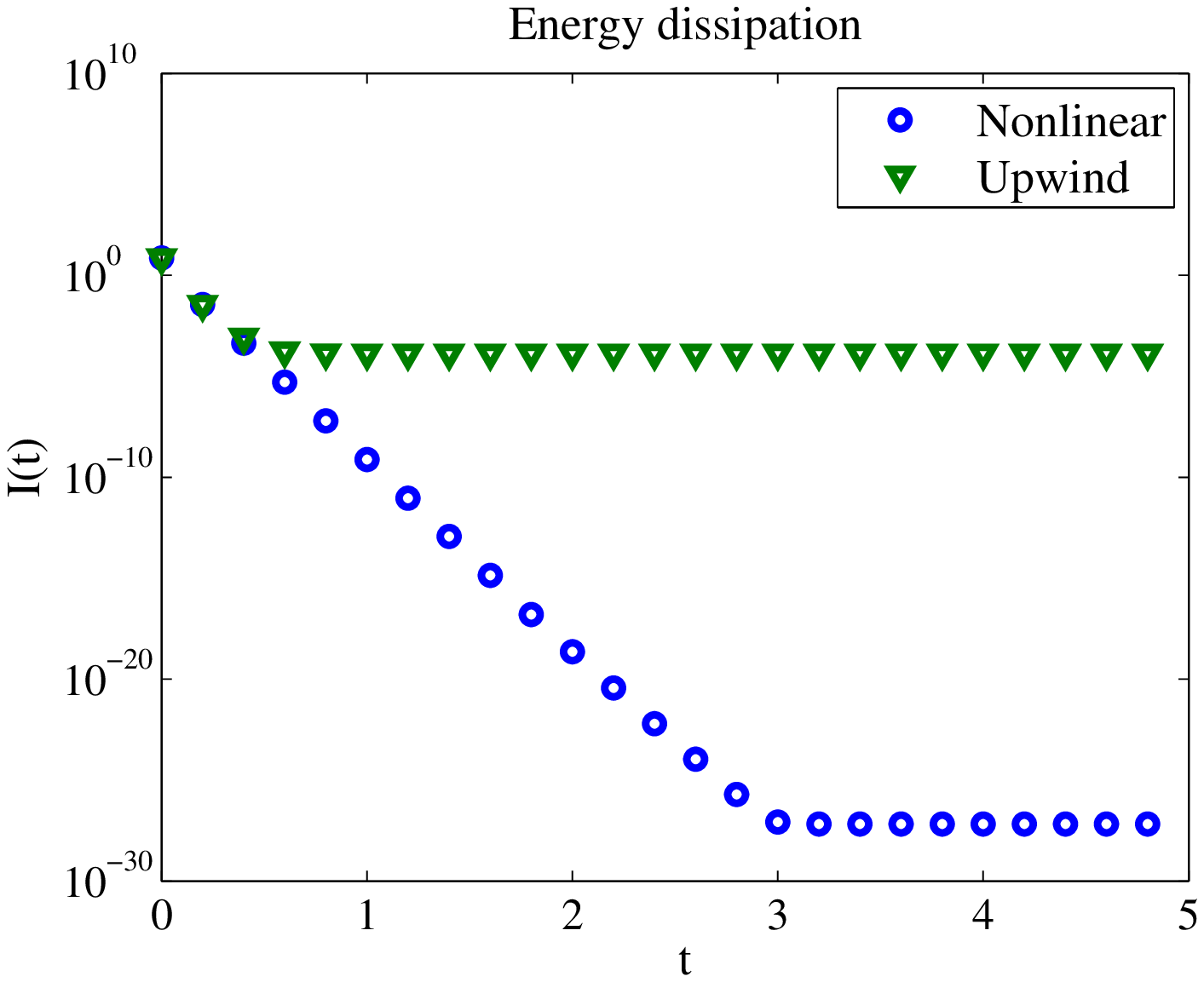}}
\caption{Nonlinear case: relative energy $\mathcal{E}^{n}$ and dissipation $\mathcal{I}^{n}$ for different schemes in log scale, with time step $\Delta t=5.10^{-4}$ and space step $\Delta x=10^{-2}$.}
\label{comparenonlin}
\end{figure}

\paragraph*{}Our aim is to propose a finite volume scheme for convection-diffusion equations with nonlinear diffusion. We will focus on preserving steady-states in order to obtain a satisfying long-time behavior of the approximate solution. The scheme proposed in \cite{Chainais-Hillairet2007} satisfies this property but because of the nonlinear discretization of the diffusive terms, it leads to solve a nonlinear system at each time step, even in the case of a linear diffusion. The idea is to extend the Scharfetter-Gummel scheme, which is only valid in the case of a linear diffusion, for convection-diffusion equations with nonlinear diffusion, even in the degenerate case. Some extensions of this scheme have already been proposed. Indeed, R. Eymard, J. Fuhrmann and K. Gärtner studied a scheme valid in the case where the convection and diffusion terms are nonlinear (see \cite{Eymard2006}), but their method leads to solve a nonlinear elliptic problem at each interface. A. Jüngel and P. Pietra proposed a scheme for the drift-diffusion model (see \cite{Juengel1995,Juengel1997}), but it is not very satisfying to reflect the large-time behavior of the solutions.


\subsection{General framework}

We will now consider the following problem:
\begin{equation}
\pa_{t}u-\text{div}(\nabla r(u)-\mathbf{q}u)=0 \text{ for } (x,t) \in \Omega \times (0,T),
\label{eq}
\end{equation}
with an initial condition 
\begin{equation}
u(x,0)=u_{0}(x) \text{ for } x \in \Omega.
\label{CI}
\end{equation}
Moreover, we will consider Dirichlet-Neumann boundary conditions. The boundary $\pa \Omega=\Gamma$ is split into two parts $\Gamma=\Gamma^{D} \cup \Gamma^{N}$ and, if we denote by $\mathbf{n}$ the outward normal to $\Gamma$, the boundary conditions are Dirichlet boundary conditions on $\Gamma^{D}$
\begin{equation}
u(x,t)=\overline{u}(x,t) \text{ for } (x,t) \in \Gamma^{D} \times (0,T),
\label{BCD}
\end{equation}
and homogeneous Neumann boundary conditions on $\Gamma^{N}$:
\begin{equation}
\nabla r(u) \cdot \mathbf{n}=0 \text{ on } \Gamma^{N} \times (0,T).
\label{BCN}
\end{equation}

\begin{remark}
We will construct the scheme and perform some numerical experiments in the case of Dirichlet-Neumann boundary conditions. However, for the analysis of the scheme, we will only consider the case of Dirichlet boundary conditions ($\pa\Omega=\Gamma^{D}=\Gamma$).
\end{remark}

We suppose that the following hypotheses are fulfilled:
\begin{description}
\item[(H1)] $\Omega$ is an open bounded connected subset of $\mathbb{R}^{d}$, with $d=1,2$ or $3$,
\item[(H2)] $\pa\Omega=\Gamma^{D}=\Gamma$, $\overline{u}$ is the trace on $\Gamma \times (0,T)$ of a function, also denoted $\overline{u}$, which is assumed to satisfy $\overline{u} \in H^{1}(\Omega \times (0,T))\cap L^{\infty}(\Omega \times (0,T))$ and $\overline{u} \geq 0\textit{ a.e.}$,
\item[(H3)] $u_{0} \in L^{\infty}(\Omega)$ and $u_{0} \geq 0 \textit{ a.e.}$,
\item[(H4)] $r \in C^{2}(\mathbb{R})$ is strictly increasing on $]0,+\infty[$, $r(0)=r'(0)=0$, with $r'(s)\geq c_{0}s^{\gamma-1}$,
\item[(H5)] $\mathbf{q} \in C^{1}(\overline{\Omega},\mathbb{R}^{d})$.
\end{description}
H. Alt, S. Luckhaus and A. Visintin, as well as J. Carrillo, studied the existence and uniqueness of a weak solution to the problem (\ref{eq})-(\ref{BCN}) in \cite{Alt1984} and \cite{Carrillo1999} respectively.

\begin{definition}
We say that $u$ is a solution to the problem (\ref{eq})-(\ref{CI})-(\ref{BCD})-(\ref{BCN}) if it verifies:
\begin{equation*}
u \in L^{\infty}(\Omega \times (0,T)),\,\ u-\overline{u} \in L^{2}(0,T;H^{1}_{0}(\Omega))
\end{equation*}
and for all $\psi \in \mathcal{D}(\Omega \times [0,T[)$,
\begin{equation}
\int_{0}^{T}\int_{\Omega}\left( u \, \pa_{t}\psi - \nabla(r(u))\cdot \nabla\psi+ u \, \mathbf{q}\cdot \nabla\psi \right) dx \, dt + \int_{\Omega}u(x,0)\,\psi(x,0)\,dx=0.
\label{defsol}
\end{equation}
\label{defisol}
\end{definition}

\paragraph*{} The outline of the paper is the following. In Section 2, we construct the finite volume scheme. In Section 3, we prove the existence and uniqueness of the solution of the scheme and give some estimates on this solution. Then, thanks to these estimates, we prove in Section 4 the compactness of a family of approximate solutions. It yields the convergence (up to a subsequence) of the solution $u_{\delta}$ of the scheme to a solution of (\ref{eq})-(\ref{BCN}) when $\delta$ goes to 0. 
In the last section, we present some numerical results that show the efficiency of the scheme.


\section{Presentation of the numerical scheme} 

In this section, we present our new finite volume scheme for equation (\ref{eq}) and other existing schemes. We will then compare these schemes to our new one.

\subsection{Definition of the finite volume scheme}

We first define the space discretization of $\Omega$. A regular and admissible mesh of $\Omega$ is given by a family $\mathcal{T}$ of control volumes (open and convex polygons in 2-D, polyhedra in 3-D), a family $\mathcal{E}$ of edges in 2-D (faces in 3-D) and a family of points $(x_{K})_{K \in \mathcal{T}}$ which satisfy Definition 5.1 in \cite{Eymard2000}. It implies that the straight line between two neighboring centers of cells $(x_{K},x_{L})$ is orthogonal to the edge $\sigma=K|L$. \\
In the set of edges $ \mathcal{E}$, we distinguish the interior edges $\sigma \in \mathcal{E}_{int}$ and the boundary edges $\sigma \in \mathcal{E}_{ext}$. Because of the Dirichlet-Neumann boundary conditions, we split $\mathcal{E}_{ext}$ into $ \mathcal{E}_{ext}=\mathcal{E}_{ext}^{D} \cup \mathcal{E}_{ext}^{N}$ where $ \mathcal{E}_{ext}^{D}$ is the set of Dirichlet boundary edges and $ \mathcal{E}_{ext}^{N}$ is the set of Neumann boundary edges. For a control volume $K \in \mathcal{T}$, we denote by $\mathcal{E}_{K}$ the set of its edges, $\mathcal{E}_{int,K}$ the set of its interior edges, $\mathcal{E}_{ext,K}^{D}$ the set of edges of $K$ included in $\Gamma^{D}$ and $\mathcal{E}_{ext,K}^{N}$ the set of edges of $K$ included in $\Gamma^{N}$.\\
The size of the mesh is defined by 
\begin{equation*}
\Delta x=\max_{K \in \mathcal{T}}(\text{diam}(K)).
\end{equation*}
In the sequel, we denote by d the distance in $\mathbb{R}^{d}$ and m the measure in $\mathbb{R}^{d}$ or $ \mathbb{R}^{d-1}$. \\
We note for all $\sigma \in \mathcal{E}$
\begin{equation*}
d_{\sigma}=\left\{\begin{array}{lll} \text{d}(x_{K},x_{L}), & \text{ for } \sigma \in \mathcal{E}_{int},& \sigma=K|L,\\ \text{d}(x_{K},\sigma), & \text{ for } \sigma \in \mathcal{E}_{ext,K}.\end{array}\right.
\end{equation*}
For all $\sigma \in \mathcal{E}$, we define the transmissibility coefficient $\ds{\tau_{\sigma}=\frac{\text{m}(\sigma)}{d_{\sigma}}.}$\\
For $\sigma \in \mathcal{E}_{K}$, $\mathbf{n}_{K,\sigma}$ is the unit vector normal to $\sigma$ outward to $K$.\\
We may now define the finite volume approximation of the equation (\ref{eq})-(\ref{BCN}).\\
Let $(\mathcal{T},\mathcal{E},(x_{K})_{K \in \mathcal{T}})$ be an admissible discretization of $\Omega$ and let us define the time step $\Delta t$, $N_{T}=E(T/\Delta t)$ and the increasing sequence $(t^{n})_{0\leq n \leq N_{T}}$, where $t^{n}=n\Delta t$, in order to get a space-time discretization $\mathcal{D}$ of $\Omega \times (0,T)$. The size of the space-time discretization $\mathcal{D}$ is defined by:
\begin{equation*}
\delta =\max(\Delta x, \Delta t).
\end{equation*}
First of all, the initial condition is discretized by:
\begin{equation}
U^{0}_{K}=\frac{1}{\text{m}(K)}\int_{K}u_{0}(x)\,dx, \quad K \in \mathcal{T}.
\label{CId}
\end{equation}
In order to introduce the finite volume scheme, we also need to define the numerical boundary conditions:
\begin{equation}
U^{n+1}_{\sigma}=\frac{1}{\Delta t \, \text{m}(\sigma)}\int_{t^{n}}^{t^{n+1}}\int_{\sigma}\overline{u}(s,t)\,ds \, dt, \quad \sigma \in \mathcal{E}_{ext}^{D}, \,\ n\geq 0.
\label{BCd}
\end{equation}
We set 
\begin{equation}
q_{K,\sigma}=\frac{1}{\text{m}(\sigma)}\int_{\sigma}\mathbf{q}(x)\cdot \mathbf{n}_{K,\sigma}\,ds(x), \quad \forall K \in \mathcal{T}, \quad \forall \sigma \in \mathcal{E}_{K}.
\label{q}
\end{equation}
The finite volume scheme is obtained by integrating the equation (\ref{eq}) on each control volume and by using the divergence theorem. We choose a backward Euler discretization in time (in order to avoid a restriction on the time step of the form $\Delta t=O(\Delta x^{2})$). Then the scheme on $u$ is given by the following set of equations:
\begin{equation}
\text{m}(K)\frac{\uknp-\ukn}{\Delta t}+\sum_{\sigma \in \mathcal{E}_{K}} \mathcal{F}_{K,\sigma}^{n+1}=0,
\label{schemegene}
\end{equation}
where the numerical flux $\mathcal{F}_{K,\sigma}^{n+1}$ is an approximation of $\ds{-\int_{\sigma} (\nabla r(u)-\mathbf{q}u)\cdot \mathbf{n}_{K,\sigma}}$ which remains to be defined.


\subsection{Definition of the numerical flux}

\subsubsection{Existing schemes}

We presented in introduction some numerical results obtained with different choices of numerical fluxes for the drift-diffusion system. We are now going to define precisely these fluxes.\\
\newline
\textbf{The classical upwind flux.} This flux was studied in \cite{Eymard2000} for a scalar convection-diffusion equation. It is valid both in the case of a linear diffusion and in the case of a nonlinear diffusion. The diffusion term is discretized classically by using a two-points flux and the convection term is discretized with the upwind flux, whose origin can be traced back to the work of R. Courant, E. Isaacson and M. Rees \cite{Courant1952}. This flux was then used for the drift-diffusion system for semiconductors in \cite{Chainais-Hillairet2003a} and \cite{Chainais-Hillairet2003,Chainais-Hillairet2004} in 1-D and in 2-D respectively. The definition of this flux is
\begin{equation}
\mathcal{F}_{K,\sigma}^{n+1}= \left\{\begin{array}{lll} \tau_{\sigma} \left(r\left(U_{K}^{n+1}\right)-r\left(U_{L}^{n+1}\right)+d_{\sigma}\left(q_{K,\sigma}^{+}U_{K}^{n+1}-q_{K,\sigma}^{-}U_{L}^{n+1}\right)\right), & &  \forall \sigma=K|L \in \mathcal{E}_{int,K},\\ \tau_{\sigma} \left(r\left(U_{K}^{n+1}\right)-r\left(U_{\sigma}^{n+1}\right)+d_{\sigma}\left(q_{K,\sigma}^{+}U_{K}^{n+1}-q_{K,\sigma}^{-}U_{\sigma}^{n+1}\right)\right),& &  \forall \sigma \in \mathcal{E}_{ext,K}^{D},\\ 0, & &  \forall \sigma \in \mathcal{E}_{ext,K}^{N}, \end{array}\right.
\label{upwindclass}
\end{equation}
where $s^{+}=\max(s,0)$ and $s^{-}=\max(-s,0)$ are the positive and negative parts of a real number $s$. \\
\textbf{The upwind flux with nonlinear discretization of the diffusion term.} This flux was introduced in \cite{Chainais-Hillairet2007} in the context of the drift-diffusion system for semiconductors. The idea is to write the flux $\ds{-\int_{\sigma} (\nabla r(u)-\textbf{q}u)\cdot \mathbf{n}_{K,\sigma}}$ as $\ds{-\int_{\sigma} (u \nabla h(u)-\textbf{q}u)\cdot \mathbf{n}_{K,\sigma}}$, where $h$ is the enthalpy function defined by (\ref{h}). The flux is then defined with a standard upwinding for the convective term and a nonlinear approximation for the diffusive term:
\begin{eqnarray*}
& &\mathcal{F}_{K,\sigma}^{n+1}= \\
& &\left\{\begin{array}{lll}-\tau_{\sigma}\left(\min\left(U_{K}^{n+1},U_{L}^{n+1}\right)Dh\left(U^{n+1}\right)_{K,\sigma}+d_{\sigma}\left(q_{K,\sigma}^{+}U_{K}^{n+1}-q_{K,\sigma}^{-}U_{L}^{n+1}\right)\right), & & \forall \sigma=K|L,\\ -\tau_{\sigma}\left(\min\left(U_{K}^{n+1},U_{\sigma}^{n+1}\right)Dh\left(U^{n+1}\right)_{K,\sigma}+d_{\sigma}\left(q_{K,\sigma}^{+}U_{K}^{n+1}-q_{K,\sigma}^{-}U_{\sigma}^{n+1}\right)\right), & & \forall \sigma \in \mathcal{E}_{ext,K}^{D}, \\ 0, & &  \forall \sigma \in \mathcal{E}_{ext,K}^{N}, \end{array}\right.
\end{eqnarray*}
where for a given function $f$, $Df(U)_{K,\sigma}$ is defined by
\begin{equation*}
Df(U)_{K,\sigma}=\left\{\begin{array}{lll} f(U_{L})-f(U_{K}), & & \text{ if } \sigma=K|L \in \mathcal{E}_{K,int}, \\ f(U_{\sigma})-f(U_{K}), & & \text{ if } \sigma \in \mathcal{E}_{K,ext}^{D}, \\ 0, & & \text{ if } \sigma \in \mathcal{E}_{K,ext}^{N}.\end{array}\right.
\end{equation*}
This flux preserves the thermal equilibrium and it is proved that the numerical solution converges to this equilibrium when time goes to infinity.\\
\textbf{The Scharfetter-Gummel flux.} This flux is widely used in the semiconductors framework in the case of a linear diffusion, namely $r(s)=s$. It has been proposed by D.L. Scharfetter and H.K. Gummel in \cite{Scharfetter1969} for the numerical approximation of the one-dimensional drift-diffusion model. We also refer to the work of A.M. Il'in \cite{Il'in1969}, where the same kind of flux was introduced for one-dimensional finite-difference schemes. The Scharfetter-Gummel flux preserves steady-state, and is second order accurate in space (see \cite{Lazarov1996}). It is defined by:
\begin{equation*}
\mathcal{F}_{K,\sigma}^{n+1}=\left\{\begin{array}{lll}\tau_{\sigma}\left(B(-d_{\sigma}q_{K,\sigma})U_{K}^{n+1}-B(d_{\sigma}q_{K,\sigma})U_{L}^{n+1}\right), & &  \forall \sigma=K|L \in \mathcal{E}_{K,int},\\ \tau_{\sigma}\left(B(-d_{\sigma}q_{K,\sigma})U_{K}^{n+1}-B(d_{\sigma}q_{K,\sigma})U_{\sigma}^{n+1}\right), & &  \forall \sigma \in \mathcal{E}_{K,ext}^{D},\\ 0, & &  \forall \sigma \in \mathcal{E}_{K,ext}^{N},\end{array}\right.
\end{equation*}
where $B$ is the Bernoulli function defined by
\begin{equation*}
B(x)=\frac{x}{e^{x}-1} \text{ for } x \neq 0, \quad B(0)=1.
\end{equation*}


\subsubsection{Extension of the Scharfetter-Gummel flux}

Now we will extend the Scharfetter-Gummel flux to the case of a nonlinear diffusion. Firstly, if we consider the linear case with a viscosity coefficient $\varepsilon>0$, namely 
\begin{equation*}
\pa_{t}u-\text{div}(\varepsilon \nabla u-\mathbf{q}u)=0 \text{ for } (x,t) \in \Omega \times (0,T),
\end{equation*}
then the Scharfetter-Gummel flux is defined by:
\begin{equation}
\mathcal{F}_{K,\sigma}^{n+1}= \tau_{\sigma} \varepsilon \left(B\left(\frac{-d_{\sigma}q_{K,\sigma}}{\varepsilon}\right)U_{K}^{n+1}-B\left(\frac{d_{\sigma}q_{K,\sigma}}{\varepsilon}\right)U_{L}^{n+1}\right) \quad \forall \sigma=K|L \in \mathcal{E}_{int,K}.
\label{sgvisc}
\end{equation}
Using the following properties of the Bernoulli function:
\begin{equation*}
B(s) \underset{s \rightarrow +\infty}{\longrightarrow} 0 \text{ and } B(s) \underset{-\infty}{\sim} -s,
\end{equation*}
it is clear that if $\varepsilon$ tends to zero, this flux degenerates into the classical upwind flux for the transport equation $\pa_{t}u-\text{div}(\textbf{q}u)=0$:
\begin{equation}
\mathcal{F}_{K,\sigma}^{n+1}= \text{m}(\sigma)\left(q_{K,\sigma}^{+}U_{K}^{n+1}-q_{K,\sigma}^{-}U_{L}^{n+1}\right) \quad \forall \sigma=K|L \in \mathcal{E}_{int,K}.
\label{upwindtransport}
\end{equation}
Now considering a nonlinear diffusion, we can write $\nabla r(u)$ as $r'(u) \nabla u$. We denote by $dr_{K,\sigma}$ an approximation of $r'(u)$ at the interface $\sigma \in \mathcal{E}_{K}$, which will be defined later. We consider this term as a viscosity coefficient and then, using (\ref{sgvisc}), we extend the Scharfetter-Gummel flux by defining:
\begin{equation}
\mathcal{F}_{K,\sigma}^{n+1}=\left\{\begin{array}{lll} \ds{\tau_{\sigma} dr_{K,\sigma} \left(B\left(\frac{-d_{\sigma}q_{K,\sigma}}{dr_{K,\sigma}}\right)U_{K}^{n+1}-B\left(\frac{d_{\sigma}q_{K,\sigma}}{dr_{K,\sigma}}\right)U_{L}^{n+1}\right),}& &  \forall \sigma=K|L \in \mathcal{E}_{int,K}, \\ \ds{\tau_{\sigma} dr_{K,\sigma} \left(B\left(\frac{-d_{\sigma}q_{K,\sigma}}{dr_{K,\sigma}}\right)U_{K}^{n+1}-B\left(\frac{d_{\sigma}q_{K,\sigma}}{dr_{K,\sigma}}\right)U_{\sigma}^{n+1}\right),}& &  \forall \sigma \in \mathcal{E}_{ext,K}^{D}, \\ 0, & & \forall \sigma \in \mathcal{E}_{ext,K}^{N}. \end{array}\right.
\label{sgext}
\end{equation}
In the degenerate case, $dr_{K,\sigma}$ can vanish and then this flux degenerates into the upwind flux (\ref{upwindtransport}). Now it remains to define $dr_{K,\sigma}$.\\
\newline
\textbf{Definition of $dr_{K,\sigma}$.} A first possibility is to take the value of $r'$ at the average of $U_{K}$ and $U_{\sigma}$:
\begin{equation}
dr_{K,\sigma}=\left\{\begin{array}{lll} \ds{r'\left(\frac{U_{K}+U_{L}}{2}\right),} & & \forall \sigma=K|L \in \mathcal{E}_{int,K}, \\ \ds{r'\left(\frac{U_{K}+U_{\sigma}}{2}\right),} & & \forall \sigma \in \mathcal{E}_{ext,K}^{D}. \end{array}\right.
\label{dr1}
\end{equation}
This choice is quite close to the one of A. Jüngel and P. Pietra (see \cite{Juengel1995,Juengel1997}). However, considering the numerical results presented in the introduction, it seems important that the numerical flux preserves the equilibrium. Therefore, we define the function $dr$ as follows: for $a$, $b \in \mathbb{R}_{+}$,
\begin{equation}
dr(a,b)=\left\{ \begin{array}{rll} \ds{\frac{h(b)-h(a)}{\log(b)-\log(a)}} & &\text{ if } ab>0 \text{ and } a \neq b,\\ \ds{r'\left(\frac{a+b}{2}\right)} & &\text{ elsewhere, } \end{array} \right.
\label{defdr}
\end{equation}
and we set for all $K \in \mathcal{T}$
\begin{equation}
dr_{K,\sigma}=\left\{\begin{array}{lll} dr(U_{K},U_{L}), & & \text{ for } \sigma=K|L \in \mathcal{E}_{K,int}, \\ dr(U_{K},U_{\sigma}), & & \text{ for } \sigma \in \mathcal{E}_{K,ext}^{D}.\end{array}\right.
\label{dr2}
\end{equation}

\begin{remark}
Let $K \in \mathcal{T}$ and $\sigma \in \mathcal{E}_{K}$. We assume that $dr_{K,\sigma}$ is defined by (\ref{dr2}) in (\ref{sgext}) and that $U_{K}>0$ and $U_{\sigma}>0$. If $d_{\sigma}q_{K,\sigma}=Dh(U)_{K,\sigma}$, then $\mathcal{F}_{K,\sigma}=0$.\\
Indeed, 
\begin{eqnarray*}
\mathcal{F}_{K,\sigma} & = & \tau_{\sigma}dr_{K,\sigma}\left(B\left(-\frac{Dh(U)_{K,\sigma}}{dr_{K,\sigma}}\right)U_{K}-B\left(\frac{Dh(U)_{K,\sigma}}{dr_{K,\sigma}}\right)U_{\sigma}\right)\\
&=& \tau_{\sigma}Dh(U)_{K,\sigma}\left(\frac{\exp\left(\ds{\frac{Dh(U)_{K,\sigma}}{dr_{K,\sigma}}}\right)U_{K}-U_{\sigma}}{\exp\left(\ds{\frac{Dh(U)_{K,\sigma}}{dr_{K,\sigma}}}\right)-1}\right).
\end{eqnarray*}
But using the definition (\ref{defdr}) of $dr$, we obtain
\begin{equation*}
\exp\left(\ds{\frac{Dh(U)_{K,\sigma}}{dr_{K,\sigma}}}\right)=\frac{U_{\sigma}}{U_{K}},
\end{equation*}
and then $ \mathcal{F}_{K,\sigma}=0$. Thus the scheme preserves this type of steady-state.
\end{remark}
\textbf{Time discretization.} We choose an explicit expression of $dr_{K,\sigma}$:
\begin{equation}
dr_{K,\sigma}^{n}=\left\{\begin{array}{lll} dr(U_{K}^{n},U_{L}^{n}), & & \text{ for } \sigma=K|L \in \mathcal{E}_{K,int}, \\ dr(U_{K}^{n},U_{\sigma}^{n}), & & \text{ for } \sigma \in \mathcal{E}_{K,ext}^{D}.\end{array}\right.
\label{dr3}
\end{equation}
Thus we obtain a scheme which leads only to solve a linear system of equations at each time step.\\
To sum up, our extension of the Scharfetter-Gummel flux is defined by
\begin{equation}
\mathcal{F}_{K,\sigma}^{n+1}=\left\{\begin{array}{lcl} \tau_{\sigma} dr_{K,\sigma}^{n} \left(B\left(\ds{\frac{-d_{\sigma}q_{K,\sigma}}{dr_{K,\sigma}^{n}}}\right)U_{K}^{n+1}-B\left(\ds{\frac{d_{\sigma}q_{K,\sigma}}{dr_{K,\sigma}^{n}}}\right)U_{L}^{n+1}\right), & & \forall \sigma=K|L \in \mathcal{E}_{K,int}, \\ \tau_{\sigma} dr_{K,\sigma}^{n} \left(B\left(\ds{\frac{-d_{\sigma}q_{K,\sigma}}{dr_{K,\sigma}^{n}}}\right)U_{K}^{n+1}-B\left(\ds{\frac{d_{\sigma}q_{K,\sigma}}{dr_{K,\sigma}^{n}}}\right)U_{\sigma}^{n+1}\right), & & \forall \sigma \in \mathcal{E}_{K,ext}^{D}, \\ 0, & & \forall \sigma \in \mathcal{E}_{K,ext}^{N},\end{array}\right.
\label{flux}
\end{equation}
where $dr_{K,\sigma}^{n}$ is defined by (\ref{dr3}). This flux preserves the equilibrium.


\subsection{Consistency of the numerical flux}

\begin{lemma}
Let $a$, $b \in \mathbb{R}$, $a,b \geq 0$. Then there exists $\eta \in [\min(a,b),\max(a,b)]$ such that 
\begin{equation*}
dr(a,b)=r'(\eta).
\end{equation*}
\label{propdr}
\end{lemma}

\begin{proof}
The result is clear if $ab=0$ or $a=b$. Let us suppose that $ab>0$ and $a < b$ (the proof is the same if $a>b$). If we consider the change of variables $x=\log(a)$ and $y=\log(b)$, we obtain
\begin{equation*}
dr(a,b)=\frac{h(\exp(y))-h(\exp(x))}{y-x}
\end{equation*}
and using Taylor's formula, there exists $\theta \in [x,y]$ such that 
\begin{equation*}
dr(a,b)=\exp(\theta)h'(\exp(\theta))=r'(\exp(\theta)) \text{ (using the definition of $h$). }
\end{equation*}
Finally, there exists $\eta=\exp(\theta) \in [a,b]$ such that 
\begin{equation*}
dr(a,b)=r'(\eta).
\end{equation*}
\end{proof}

\begin{remark}
The flux (\ref{flux}) can also be written as
\begin{equation}
\mathcal{F}_{K,\sigma}^{n+1}=\text{m}(\sigma)q_{K,\sigma}\frac{U_{K}^{n+1}+U^{n+1}_{\sigma}}{2}-\frac{\text{m}(\sigma)q_{K,\sigma}}{2}\coth \left(\frac{d_{\sigma}q_{K,\sigma}}{2dr^{n}_{K,\sigma}}\right)(U^{n+1}_{\sigma}-U^{n+1}_{K}).
\label{flux2}
\end{equation}
The first term is a centred discretization of the convective part. The second term is consistent with the diffusive part of equation (\ref{eq}), since $\ds{\coth(x)\underset{0}{\sim}\frac{1}{x}}$.
\end{remark}


\section{Properties of the scheme}

\subsection{Well-posedness of the scheme}

The following proposition gives the existence and uniqueness result of the solution to the scheme defined by (\ref{CId})-(\ref{BCd})-(\ref{schemegene})-(\ref{flux}) and an $L^{\infty}$-estimate on this solution. 

\begin{proposition}
Let us assume hypotheses (H1)-(H5). Let $\mathcal{D}$ be an admissible discretization of $\Omega \times (0,T)$. Then there exists a unique solution $\lbrace U_{K}^{n},K \in \mathcal{T},0\leq n \leq N_{T}\rbrace$ to the scheme (\ref{CId})-(\ref{BCd})-(\ref{schemegene})-(\ref{flux}), with $U_{K}^{n} \geq 0$ for all $K \in \mathcal{T}$ and $0\leq n \leq N_{T}$.\\
Moreover, if we suppose that the two following assumptions are fulfilled:
\begin{description}
\item[(H6)] $ \emph{div} (\mathbf{q})=0$,
\item[(H7)] there exist two constants $m>0$ and $M>0$ such that $m\leq \overline{u},u_{0}\leq M$,
\end{description}
then we have
\begin{equation}
0 <m \leq U^{n}_{K} \leq M, \quad \forall K \in \mathcal{T}, \quad \forall n \geq 0.
\label{estimlinf}
\end{equation}
\end{proposition}

\begin{proof}
At each time step, the scheme (\ref{CId})-(\ref{BCd})-(\ref{schemegene})-(\ref{flux}) leads to a system of card$(\mathcal{T})$ linear equations on $U^{n+1}=(U^{n+1}_{K})_{K \in \mathcal{T}}$ which can be written:
\begin{equation*}
A^{n}U^{n+1}=S^{n},
\end{equation*}
where :
\begin{itemize}
\item[$\bullet$] $A^{n}$ is the matrix defined by
\begin{alignat*}{3}
&A^{n}_{K,K} & = &  \frac{\text{m}(K)}{\Delta t}+\sum_{\sigma \in \mathcal{E}_{K}}\tau_{\sigma}dr^{n}_{K,\sigma}B\left(\frac{-d_{\sigma}q_{K,\sigma}}{dr^{n}_{K,\sigma}}\right) \,\ \forall K \in \mathcal{T},\\
&A^{n}_{K,L} & = & -\tau_{\sigma}dr^{n}_{K,\sigma}B\left(\frac{d_{\sigma}q_{K,\sigma}}{dr^{n}_{K,\sigma}}\right) \,\ \forall L \in \mathcal{T} \text{ such that } \sigma=K|L \in \mathcal{E}_{int,K};
\end{alignat*}
\item[$\bullet$] $\ds{S^{n}=\left(\frac{\text{m}(K)}{\Delta t}U^{n}_{K}\right)_{K \in \mathcal{T}}+Tb^{n}}$, with
\begin{equation*}
Tb^{n}_{K} = \left\{\begin{array}{ll} 0  & \text{ if } K \in \mathcal{T} \text{ such that } \text{m}(\pa K\cap \Gamma)=0,\\ \ds{\sum_{\sigma \in \mathcal{E}_{ext,K}^{D}}\tau_{\sigma}dr^{n}_{K,\sigma}B\left(\ds{\frac{d_{\sigma}q_{K,\sigma}}{dr^{n}_{K,\sigma}}}\right)U^{n+1}_{\sigma}}  & \text{ if } K \in \mathcal{T} \text{ such that } \text{m}(\pa K\cap \Gamma)\neq 0. \end{array}\right.
\end{equation*}
\end{itemize}
The diagonal terms of $A^{n}$ are positive and the offdiagonal terms are nonnegative (since $B(x)>0$ for all $x \in \mathbb{R}$ and $dr^{n}_{K,\sigma} \geq 0$ for all $K \in \mathcal{T}$, for all $\sigma \in \mathcal{E}_{K}$). Moreover, since $dr^{n}_{K,\sigma}=dr_{L,\sigma}^{n}$ and $q_{K,\sigma}=-q_{L,\sigma}$ for all $ \sigma=K|L \in \mathcal{E}_{int}$, we have for all $L \in \mathcal{T}$:
\begin{equation*}
\left|A^{n}_{L,L}\right|-\sum_{\substack{K \in \mathcal{T} \\ K \neq L}}\left|A^{n}_{K,L}\right|=\frac{\text{m}(L)}{\Delta t}>0,
\end{equation*}
and then $A^{n}$ is strictly diagonally dominant with respect to the columns. $A^{n}$ is then an M-matrix so $A^{n}$ is invertible, which gives existence and uniqueness of the solution of the scheme. Moreover, $(A^{n})^{-1}\geq 0$ and since $U_{K}^{0} \geq 0$ for all $K\in \mathcal{T}$ (using (H3)) and $U^{n+1}_{\sigma} \geq 0$ for all $n \geq 0$, for all $\sigma \in \mathcal{E}_{ext}^{D}$ (using (H2)), it is easy to prove by induction that $U^{n}_{K} \geq 0$ for all $K \in \mathcal{T}$, for all $n \geq 0$.\\
Now, we suppose that (H6) and (H7) are fulfilled. We prove that $U^{n}_{K}\leq M $ for all $K \in \mathcal{T}$, for all $n \geq 0$ by induction. Thanks to hypothesis (H7), we have clearly $U^{0}_{K}\leq M $ for all $ K \in \mathcal{T}$.\\
Let us suppose that $U^{n}_{K}\leq M \quad \forall K \in \mathcal{T}$. We want to prove $U^{n+1}_{K}\leq M \quad \forall K \in \mathcal{T}$.\\
Let us define $\textbf{M}=(M,...,M)^{T} \in \mathbb{R}^{\text{card}(\mathcal{T})}$. Since $A^{n}$ is an M-matrix, we have $(A^{n})^{-1} \geq 0$ and then it suffices to prove that $A^{n}\left(U^{n+1}-\textbf{M}\right) \leq 0$. \\
We first compute $A^{n}\textbf{M}$. Using the following property of the Bernoulli function:
\begin{equation}
B(x)-B(-x)=-x \quad \forall x \in \mathbb{R},
\label{propriB}
\end{equation}
we obtain that for all $K \in \mathcal{T}$,
\begin{equation*}
\left(A^{n}\textbf{M}\right)_{K}=M\left(\frac{\text{m}(K)}{\Delta t}+\sum_{\sigma \in \mathcal{E}_{int,K}}\text{m}(\sigma)q_{K,\sigma}+\sum_{\sigma \in \mathcal{E}_{ext,K}^{D}}\tau_{\sigma}dr^{n}_{K,\sigma}B\left(-\frac{d_{\sigma}q_{K,\sigma}}{dr^{n}_{K,\sigma}}\right)\right).
\end{equation*}
Then we compute $A^{n}\left(U^{n+1}-\textbf{M}\right)$: for all $K \in \mathcal{T}$
\begin{eqnarray*}
\left(A^{n}\left(U^{n+1}-\textbf{M}\right)\right)_{K} &=& \frac{\text{m}(K)}{\Delta t}(U^{n}_{K}-M)+ \sum_{\sigma \in \mathcal{E}_{ext,K}^{D}}\tau_{\sigma}dr^{n}_{K,\sigma}B\left(\frac{d_{\sigma}q_{K,\sigma}}{dr^{n}_{K,\sigma}}\right)U^{n+1}_{\sigma} \\
& & -M \sum_{\sigma \in \mathcal{E}_{int,K}}\text{m}(\sigma)q_{K,\sigma}-M\sum_{\sigma \in \mathcal{E}_{ext,K}^{D}}\tau_{\sigma}dr^{n}_{K,\sigma}B\left(-\frac{d_{\sigma}q_{K,\sigma}}{dr^{n}_{K,\sigma}}\right).
\end{eqnarray*}
By induction hypothesis, the first term is nonpositive. Moreover, using hypothesis (H7) and the property (\ref{propriB}), we obtain
\begin{eqnarray*}
\left(A^{n}\left(U^{n+1}-\textbf{M}\right)\right)_{K} &\leq &-M \sum_{\sigma \in \mathcal{E}_{int,K}}\text{m}(\sigma)q_{K,\sigma} -M \sum_{\sigma \in \mathcal{E}_{ext,K}^{D}}\text{m}(\sigma)q_{K,\sigma}\\
&\leq & -M \sum_{\sigma \in \mathcal{E}_{K}}\text{m}(\sigma)q_{K,\sigma}.
\end{eqnarray*}
However, using hypothesis (H6) and the definition of $q_{K,\sigma}$ (\ref{q}), we get
\begin{equation*}
\sum_{\sigma \in \mathcal{E}_{K}}\text{m}(\sigma)q_{K,\sigma}=\sum_{\sigma \in \mathcal{E}_{K}}\int_{\sigma}q \cdot \mathbf{n}_{K,\sigma}\, ds= \int_{K}\text{div}(q)=0,
\end{equation*}
and then $\left(A^{n}\left(U^{n+1}-\textbf{M}\right)\right)_{K} \leq 0$ for all $K \in \mathcal{T}$.\\
So we have $A^{n}\left(U^{n+1}-\textbf{M}\right) \leq 0$, therefore we deduce that $U^{n+1}-\textbf{M} \leq 0$, hence $U^{n+1}_{K} \leq M \quad \forall K$ and we can show by the same way that $U^{n+1}_{K} \geq m \quad \forall K$.
\end{proof}

\begin{remark}
In the case of the drift-diffusion system for semiconductors, the hypothesis (H6) is not fulfilled ($ \Delta V \neq 0$). Nevertheless, if we assume that
\begin{itemize}
\item the doping profile $C$ is equal to $0$,
\item there exist two constants $m>0$ and $M>0$ such that $m\leq \overline{N},N_{0},\overline{P},P_{0}\leq M$,
\item $M \Delta t \leq 1$,
\end{itemize}
then we have, using the same kind of proof as in \cite{Chainais-Hillairet2003},
\begin{eqnarray*}
0 <m \leq N^{n}_{K} \leq M, & &\quad \forall K \in \mathcal{T}, \quad \forall n \geq 0,\\
0 <m \leq P^{n}_{K} \leq M, & &\quad \forall K \in \mathcal{T}, \quad \forall n \geq 0.
\end{eqnarray*}
\end{remark}

\begin{definition}
Let $ \mathcal{D}$ be an admissible discretization of $ \Omega \times (0,T)$. The approximate solution to the problem (\ref{eq})-(\ref{CI})-(\ref{BCD})-(\ref{BCN}) associated to the discretization $\mathcal{D}$ is defined as piecewise constant function by:
\begin{equation}
u_{\delta}(x,t) =  U^{n+1}_{K},  \quad \forall(x,t)\in K\times[t^{n},t^{n+1}[, \label{defiuapp}
\end{equation}
where $ \lbrace U^{n}_{K},K \in \mathcal{T},0\leq n \leq N_{T}\rbrace $ is the unique solution to the scheme (\ref{CId})-(\ref{BCd})-(\ref{schemegene})-(\ref{flux}).
\end{definition}


\subsection{Discrete $L^{2}\left(0,T;H^{1}\right)$ estimate on $u_{\delta}$}

In this section, we prove a discrete $L^{2}\left(0,T;H^{1}\right)$ estimate on $u_{\delta}$ in the nondegenerate case, which leads to compactness and convergence results.\\
For a piecewise constant function $v_{\delta}$ defined by $v_{\delta}(x,t)=v_{K}^{n+1}$ for $(x,t) \in K \times [t^{n},t^{n+1}[$ and $v_{\delta}(\gamma,t)=v_{\sigma}^{n+1}$ for $(\gamma,t) \in \sigma \times [t^{n},t^{n+1}[$, we define
\begin{equation*}
\Vert v_{\delta} \Vert_{1,\mathcal{D}}^{2}=\sum_{n=0}^{N_{T}}\Delta t\left( \sum_{\substack{\sigma \in \mathcal{E}_{int} \\ \sigma=K|L}} \tau_{\sigma} \left\vert v^{n+1}_{L}-v^{n+1}_{K}\right\vert^{2} + \sum_{K \in \mathcal{T}}\sum_{\sigma \in \mathcal{E}_{ext,K}^{D}} \tau_{\sigma} \left\vert v^{n+1}_{\sigma}-v^{n+1}_{K}\right\vert^{2} \right).
\end{equation*}

\begin{proposition}
\label{propestiml2}
Let assume (H1)-(H7) are satisfied. Let $u_{\delta}$ be defined by the scheme (\ref{CId})-(\ref{BCd})-(\ref{schemegene})-(\ref{flux}) and (\ref{defiuapp}).\\
There exists $D_{1}>0$ only depending on $r$, $\textbf{q}$, $u_{0}$, $\overline{u}$, $\Omega$ and $T$ such that
\begin{equation}
\Vert u_{\delta} \Vert_{1,\mathcal{D}}^{2} \leq D_{1}.
\label{estiml2}
\end{equation}
\end{proposition}

\begin{proof}
We follow the proof of Lemma 4.2 in \cite{Eymard2006}. Throughout this proof, $D_{i}$ denotes constants which depend only on $r$, $\textbf{q}$, $u_{0}$, $\overline{u}$, $\Omega$ and $T$. We set
\begin{equation*}
\overline{U}^{n+1}_{K}=\ds{\frac{1}{\Delta t\text{m}(K)}\int_{t^{n}}^{t^{n+1}}\int_{K}\overline{u}(x,t)\, dx \, dt},\, \forall K \in \mathcal{T}, \,\ \forall n \in \mathbb{N},
\end{equation*} 
and
\begin{equation*}
w^{n+1}_{K}=U^{n+1}_{K}-\overline{U}^{n+1}_{K},\, \forall K \in \mathcal{T}, \,\ \forall n \in \mathbb{N}.
\end{equation*}
We multiply the scheme (\ref{schemegene}) by $\Delta t w^{n+1}_{K}$ and we sum over $n$ and $K$. We obtain $A+B=0$, where: 
\begin{eqnarray*}
A &=& \sum_{n=0}^{N_{T}}\sum_{K \in \mathcal{T}}\text{m}(K)\left(U^{n+1}_{K}-U^{n}_{K}\right)w_{K}^{n+1},\\
B &=& \sum_{n=0}^{N_{T}} \Delta t \sum_{K \in \mathcal{T}} \sum_{\sigma \in \mathcal{E}_{K}} \mathcal{F}^{n+1}_{K,\sigma}w^{n+1}_{K}.
\end{eqnarray*}

\textbf{Estimate of $A$.} This term is treated in \cite{Eymard2006}. We get:
\begin{equation}
A \geq -\frac{1}{2}\Vert u_{0}-\overline{u}(.,0)\Vert_{L^{2}(\Omega)}^{2}-2\Vert \pa_{t}\overline{u}\Vert_{L^{1}(\Omega \times (0,T))}|M-m| =-D_{2}. \label{estimA}
\end{equation}

\textbf{Estimate of $B$.} A discrete integration by parts yields (using that $w_{\sigma}^{n+1}=0$ for all $\sigma \in \mathcal{E}_{ext}^{D}$ and for all $n \geq 0$):

\begin{equation*}
B = \sum_{n=0}^{N_{T}} \Delta t \sum_{\substack{\sigma \in \mathcal{E}_{int} \\ \sigma=K|L}} \mathcal{F}^{n+1}_{K,\sigma}\left(w^{n+1}_{K}-w_{L}^{n+1}\right) + \sum_{n=0}^{N_{T}} \Delta t \sum_{K \in \mathcal{T}} \sum_{\sigma \in \mathcal{E}_{ext,K}^{D}} \mathcal{F}^{n+1}_{K,\sigma} \left(w^{n+1}_{K}-w_{\sigma}^{n+1}\right),
\end{equation*}
which delivers $B=B'-\overline{B}$, with:
\begin{eqnarray*}
B' &=& \sum_{n=0}^{N_{T}} \Delta t \sum_{\substack{\sigma \in \mathcal{E}_{int} \\ \sigma=K|L}} \mathcal{F}^{n+1}_{K,\sigma}\left(U^{n+1}_{K}-U_{L}^{n+1}\right) + \sum_{n=0}^{N_{T}} \Delta t \sum_{K \in \mathcal{T}} \sum_{\sigma \in \mathcal{E}_{ext,K}^{D}} \mathcal{F}^{n+1}_{K,\sigma}\left(U^{n+1}_{K}-U_{\sigma}^{n+1}\right),\\
\overline{B} &=& \sum_{n=0}^{N_{T}} \Delta t \sum_{\substack{\sigma \in \mathcal{E}_{int} \\ \sigma=K|L}} \mathcal{F}^{n+1}_{K,\sigma}\left(\overline{U}^{n+1}_{K}-\overline{U}_{L}^{n+1}\right) + \sum_{n=0}^{N_{T}} \Delta t \sum_{K \in \mathcal{T}} \sum_{\sigma \in \mathcal{E}_{ext,K}^{D}} \mathcal{F}^{n+1}_{K,\sigma}\left(\overline{U}^{n+1}_{K}-\overline{U}_{\sigma}^{n+1}\right).
\end{eqnarray*}

\textbf{Estimate of $\overline{B}$.} Using the expression (\ref{flux2}) of $\mathcal{F}_{K,\sigma}^{n+1}$, we have $\overline{B}=\overline{B}_{1}+\overline{B}_{2}$ with
\begin{eqnarray*}
\overline{B}_{1}&=& \sum_{n=0}^{N_{T}} \Delta t \sum_{\substack{\sigma \in \mathcal{E}_{int} \\ \sigma=K|L}} \frac{\text{m}(\sigma)q_{K,\sigma}}{2}\left(U_{K}^{n+1}+U_{L}^{n+1}\right)\left(\overline{U}_{K}^{n+1}-\overline{U}_{L}^{n+1}\right) \\
& & + \sum_{n=0}^{N_{T}} \Delta t \sum_{K \in \mathcal{T}} \sum_{\sigma \in \mathcal{E}_{ext,K}^{D}}\frac{\text{m}(\sigma)q_{K,\sigma}}{2}\left(U_{K}^{n+1}+U_{\sigma}^{n+1}\right)\left(\overline{U}_{K}^{n+1}-\overline{U}_{\sigma}^{n+1}\right),\\
\overline{B}_{2}&=& \sum_{n=0}^{N_{T}} \Delta t \sum_{\substack{\sigma \in \mathcal{E}_{int} \\ \sigma=K|L}} \frac{\text{m}(\sigma)q_{K,\sigma}}{2}\coth\left(\frac{d_{\sigma}q_{K,\sigma}}{2dr^{n}_{K,\sigma}}\right)\left(U_{K}^{n+1}-U_{L}^{n+1}\right)\left(\overline{U}_{K}^{n+1}-\overline{U}_{L}^{n+1}\right)\\
& & + \sum_{n=0}^{N_{T}} \Delta t \sum_{K \in \mathcal{T}} \sum_{\sigma \in \mathcal{E}_{ext,K}^{D}}\frac{\text{m}(\sigma)q_{K,\sigma}}{2}\coth\left(\frac{d_{\sigma}q_{K,\sigma}}{2dr^{n}_{K,\sigma}}\right)\left(U_{K}^{n+1}-U_{\sigma}^{n+1}\right)\left(\overline{U}_{K}^{n+1}-\overline{U}_{\sigma}^{n+1}\right).
\end{eqnarray*}
The term $\overline{B}_{1}$ is treated like in \cite{Eymard2006}, which leads to
\begin{equation*}
\vert \overline{B}_{1} \vert \leq M \Vert \textbf{q} \Vert_{\infty}\Vert \overline{u}_{\delta} \Vert_{1,\mathcal{D}}\text{dm}(\Omega)=D_{3}.
\end{equation*}
We apply Young's inequality for $\overline{B}_{2}$: for any $\alpha>0$, we have
\begin{eqnarray*}
\left\vert \overline{B}_{2} \right\vert &\leq &\frac{\alpha}{2} \sum_{n=0}^{N_{T}} \Delta t \sum_{\substack{\sigma \in \mathcal{E}_{int} \\ \sigma=K|L}} \tau_{\sigma}\left(dr_{K,\sigma}^{n}\right)^{2}\left(\frac{d_{\sigma}q_{K,\sigma}}{2dr_{K,\sigma}^{n}}\coth\left(\frac{d_{\sigma}q_{K,\sigma}}{2dr^{n}_{K,\sigma}}\right)\right)^{2}\left(U_{K}^{n+1}-U_{L}^{n+1}\right)^{2}\\
& & + \frac{\alpha}{2} \sum_{n=0}^{N_{T}} \Delta t \sum_{K \in \mathcal{T}} \sum_{\sigma \in \mathcal{E}_{ext,K}^{D}} \tau_{\sigma}\left(dr_{K,\sigma}^{n}\right)^{2} \left(\frac{d_{\sigma}q_{K,\sigma}}{2dr^{n}_{K,\sigma}}\coth\left(\frac{d_{\sigma}q_{K,\sigma}}{2dr^{n}_{K,\sigma}}\right)\right)^{2}\left(U_{K}^{n+1}-U_{\sigma}^{n+1}\right)^{2}\\
& &+ \frac{1}{2\alpha}\Vert \overline{u}_{\delta}\Vert_{1,\mathcal{D}}^{2}. 
\end{eqnarray*}
By the hypothesis (H4), we have $\ds{\inf_{s \in [m,M]}r'(s)>0}$. Then, using Lemma \ref{propdr}, the $L^{\infty}$ estimate on $u_{\delta}$ (\ref{estimlinf}) and the hypothesis (H5), we have
\begin{equation*}
\frac{d_{\sigma}q_{K,\sigma}}{2dr^{n}_{K,\sigma}} \leq \frac{\Vert \textbf{q} \Vert_{\infty}\text{diam}(\Omega)}{\ds{\inf_{s\in [m,M]}r'(s)}}, \, \forall n \in \mathbb{N},\, \forall K \in \mathcal{T}, \, \forall \sigma \in \mathcal{E}_{K}. 
\end{equation*}
Moreover, since $x \mapsto x \coth(x)$ is continuous on $\mathbb{R}$, we obtain
\begin{equation*}
\left(\frac{d_{\sigma}q_{K,\sigma}}{2dr^{n}_{K,\sigma}}\coth\left(\frac{d_{\sigma}q_{K,\sigma}}{2dr^{n}_{K,\sigma}}\right)\right)^{2} \leq D_{4}, \, \forall n \in \mathbb{N},\, \forall K \in \mathcal{T}, \, \forall \sigma \in \mathcal{E}_{K}. 
\end{equation*}
Thus we can bound $ \overline{B}$:
\begin{equation}
\left\vert \overline{B} \right\vert \leq D_{3}+\frac{\alpha}{2}D_{4}\left(\sup_{s \in [m,M]}r'(s)\right)^{2}\Vert u_{\delta} \Vert_{1,\mathcal{D}}^{2}+\frac{1}{2\alpha}\Vert \overline{u}_{\delta} \Vert_{1,\mathcal{D}}.
\label{estimBb} 
\end{equation} 

\textbf{Estimate of $B'$.} First, using the expression (\ref{flux2}) of the flux and Lemma \ref{propdr}, we have for all $n \geq 0$, for all $K \in \mathcal{T}$ and for all $\sigma=K|L \in \mathcal{E}_{int,K}$
\begin{eqnarray*}
\mathcal{F}^{n+1}_{K,\sigma}\left(U^{n+1}_{K}-U^{n+1}_{L}\right) &=& \frac{\text{m}(\sigma)q_{K,\sigma}}{2}\left(\left(U_{K}^{n+1}\right)^{2}-\left(U^{n+1}_{L}\right)^{2}\right)\\
& & + \tau_{\sigma} r'(\eta_{K,\sigma}^{n})\frac{d_{\sigma}q_{K,\sigma}}{2r'(\eta_{K,\sigma}^{n})} \coth\left(\frac{d_{\sigma}q_{K,\sigma}}{2r'(\eta_{K,\sigma}^{n})}\right)\left(U^{n+1}_{K}-U^{n+1}_{L}\right)^{2}.
\end{eqnarray*}  
Then, since $x\coth(x) \geq 1$ for all $x\in \mathbb{R}$, we get:
\begin{equation*}
\mathcal{F}^{n+1}_{K,\sigma}\left(U^{n+1}_{K}-U^{n+1}_{L}\right) \geq  \frac{\text{m}(\sigma)q_{K,\sigma}}{2}\left(\left(U_{K}^{n+1}\right)^{2}-\left(U^{n+1}_{L}\right)^{2}\right) + \tau_{\sigma}\inf_{s \in [m,M]}r'(s) \left(U^{n+1}_{K}-U^{n+1}_{L}\right)^{2}.
\end{equation*}
We obtain the same type of inequality for $\mathcal{F}^{n+1}_{K,\sigma}\left(U^{n+1}_{K}-U^{n+1}_{\sigma}\right)$. Thus we get 
\begin{eqnarray*}
B' &\geq & \inf_{s \in [m,M]}r'(s)\Vert u_{\delta}\Vert_{1,\mathcal{D}}^{2}+\sum_{n=0}^{N_{T}}\Delta t\sum_{\substack{\sigma \in \mathcal{E}_{int} \\ \sigma=K|L}} \frac{\text{m}(\sigma)q_{K,\sigma}}{2}\left(\left(U_{K}^{n+1}\right)^{2}-\left(U^{n+1}_{L}\right)^{2}\right)\\
& &+\sum_{n=0}^{N_{T}}\Delta t\sum_{K\in \mathcal{T}}\sum_{\sigma \in \mathcal{E}_{ext,K}^{D}}\frac{\text{m}(\sigma)q_{K,\sigma}}{2}\left(\left(U_{K}^{n+1}\right)^{2}-\left(U^{n+1}_{\sigma}\right)^{2}\right).
\end{eqnarray*}
Through integrating by parts and using the hypothesis (H6), we get
\begin{eqnarray*}
& &\sum_{n=0}^{N_{T}}\Delta t\sum_{\substack{\sigma \in \mathcal{E}_{int} \\ \sigma=K|L}} \frac{\text{m}(\sigma)q_{K,\sigma}}{2}\left(\left(U_{K}^{n+1}\right)^{2}-\left(U^{n+1}_{L}\right)^{2}\right)\\
&+&\sum_{n=0}^{N_{T}}\Delta t\sum_{K\in \mathcal{T}}\sum_{\sigma \in \mathcal{E}_{ext,K}^{D}}\frac{\text{m}(\sigma)q_{K,\sigma}}{2}\left(\left(U_{K}^{n+1}\right)^{2}-\left(U^{n+1}_{\sigma}\right)^{2}\right) \\
& = & -\sum_{n=0}^{N_{T}}\Delta t \sum_{K\in \mathcal{T}}\sum_{\sigma \in \mathcal{E}_{ext,K}^{D}} \frac{1}{2}\int_{\sigma}\textbf{q}(x)\cdot \textbf{n}_{K,\sigma}\, ds(x) \left(U_{\sigma}^{n+1}\right)^{2}=-D_{5},
\end{eqnarray*}
and then 
\begin{equation}
B' \geq \inf_{s \in [m,M]}r'(s)\Vert u_{\delta}\Vert_{1,\mathcal{D}}^{2}-D_{5}.
\label{estimBp}
\end{equation}

\textbf{Conclusion.} Using $A+B=0$ and estimates (\ref{estimA}), (\ref{estimBb}) and (\ref{estimBp}), we finally get for any $\alpha>0$:
\begin{equation*}
\left(\inf_{s\in [m,M]}r'(s)-\frac{\alpha}{2}D_{4}\left(\sup_{s\in [m,M]}r'(s)\right)^{2}\right)\Vert u_{\delta}\Vert_{1,\mathcal{D}}^{2} \leq D_{2}+D_{3}+D_{5}+\frac{1}{2\alpha}\Vert \overline{u}_{\delta}\Vert_{1,\mathcal{D}}^{2},
\end{equation*}
thus for $\alpha < \ds{\frac{2\ds{\inf_{s\in [m,M]}r'(s)}}{D_{4}\left(\ds{\sup_{s\in [m,M]}r'(s)}\right)^{2}}}$, we obtain $
\Vert u_{\delta} \Vert_{1,\mathcal{D}}^{2} \leq D_{1}$.
\end{proof}


\section{Convergence}

In this section, we prove the convergence of the approximate solution $u_{\delta}$ to a weak solution $u$ of the problem (\ref{eq})-(\ref{CI})-(\ref{BCD})-(\ref{BCN}). Our first goal is to prove the strong compactness of $(u_{\delta})_{\delta>0}$ in $L^{2}\left(\Omega \times ]0,T[\right)$. It comes from the criterion of strong compactness of a sequence by using estimates (\ref{estimlinf}) and (\ref{estiml2}). Then, we will prove the weak compactness in $L^{2}(\Omega \times ]0,T[)$ of an approximate gradient. Finally, we will show the convergence of the scheme.


\subsection{Compactness of the approximate solution}

The following Lemma is a classical consequence of Proposition \ref{propestiml2} and estimates of time translation for $u_{\delta}$ obtained from the scheme (\ref{CId})-(\ref{BCd})-(\ref{schemegene})-(\ref{flux}). The proof is similar to those of Lemma 4.3 and Lemma 4.7 in \cite{Eymard2000}. 

\begin{lemma}[Space and time translate estimates]
We suppose (H1)-(H7). Let $\mathcal{D}$ be an admissible discretization of $ \Omega \times (0,T)$. Let $u_{\delta}$ be defined by the scheme (\ref{CId})-(\ref{BCd})-(\ref{schemegene})-(\ref{flux}) and by (\ref{defiuapp}).\\
Let $ \hat{u}$ be defined by $\hat{u}_{\delta}=u_{\delta} \text{ a.e. on } \Omega \times (0,T)$ and $\hat{u}_{\delta}=0 \text{ a.e. on } \mathbb{R}^{d+1} \setminus \Omega \times (0,T)$.\\
Then we get the existence of $M_{2}>0$, only depending on $ \Omega$, $T$, $r$, $q$, $u_{0}$, $\overline{u}$ and not on $ \mathcal{D}$ such that 
\begin{equation}
\int_{0}^{T}\int_{\Omega}\left(\hat{u}_{\delta}(x+\eta,t)-\hat{u}_{\delta}(x,t)\right)^{2}\, dx\, dt \leq M_{2}|\eta|(|\eta|+4\delta), \quad \forall \eta \in \mathbb{R}^{d},
\label{transspace}
\end{equation}
and
\begin{equation}
\int_{0}^{T}\int_{\Omega}\left(\hat{u}_{\delta}(x,t+\tau)-\hat{u}_{\delta}(x,t)\right)^{2}\, dx \, dt \leq M_{2}|\tau|, \quad \forall \tau \in \mathbb{R}.
\label{transtime}
\end{equation}
\end{lemma}

Now, we define an approximation $\nabla^{\delta}u_{\delta}$ of the gradient of $u$. Therefore, we will define a dual mesh. For $K \in \mathcal{T}$ and $\sigma \in \mathcal{E}_{K}$, we define $T_{K,\sigma}$ as follows: 
\begin{itemize}
\item if $\sigma=K|L \in \mathcal{E}_{int,K}$, then $T_{K,\sigma}$ is the cell whose vertices are $x_{K}$, $x_{L}$ and those of $\sigma=K|L$,
\item if $\sigma \in \mathcal{E}_{ext,K}$, then $T_{K,\sigma}$ is the cell whose vertices are $x_{K}$ and those of $\sigma$.
\end{itemize}
See \cite{Chainais-Hillairet2004} for an example of construction of $T_{K,\sigma}$. Then $\left(\left(T_{K,\sigma}\right)_{\sigma \in \mathcal{E}_{K}}\right)_{K \in \mathcal{T}}$ defines a partition of $\Omega$.
The approximation $\nabla^{\delta}u_{\delta}$ is a piecewise function defined in $\Omega \times (0,T)$ by:
\begin{equation*}
\nabla^{\delta}u_{\delta}(x,t)=\left\{\begin{array}{lcl} \ds{\frac{\text{m}(\sigma)}{\text{m}(T_{K,\sigma})}\left(U^{n+1}_{L}-U^{n+1}_{K}\right) \mathbf{n}_{K,\sigma}} &  & \text{ if }  (x,t) \in T_{K,\sigma}\times [t^{n},t^{n+1}[, \,\ \sigma=K|L,\\ \ds{\frac{\text{m}(\sigma)}{\text{m}(T_{K,\sigma})}\left(U^{n+1}_{\sigma}-U^{n+1}_{K}\right) \mathbf{n}_{K,\sigma}} & & \text{ if } (x,t) \in T_{K,\sigma}\times [t^{n},t^{n+1}[, \,\ \sigma \in \mathcal{E}_{ext,K}. \end{array}\right.
\end{equation*}

\begin{proposition}
We suppose (H1)-(H7). \\
There exist subsequences of $(u_{\delta})_{\delta>0}$ and $(\nabla^{\delta}u_{\delta})_{\delta>0}$, still denoted $(u_{\delta})_{\delta>0}$ and $(\nabla^{\delta}u_{\delta})_{\delta>0}$, and a function $u \in L^{\infty}(0,T;H^{1}(\Omega))$ such that
\begin{equation*}
\begin{array}{cll} u_{\delta} \rightarrow  u & \text{ in } L^{2}(\Omega \times ]0,T[) \text{ strongly,} & \text{ as } \delta \rightarrow 0, \\
\nabla^{\delta}u_{\delta} \rightharpoonup \nabla u & \text{ in } (L^{2}(\Omega \times ]0,T[))^{d} \text{ weakly,} & \text{ as } \delta \rightarrow 0. \end{array}
\end{equation*}
\label{compaciteu}
\end{proposition}

\begin{proof}
Using estimates (\ref{transspace})-(\ref{transtime}) and applying the Riesz-Fréchet-Kolmogorov criterion of strong compactness \cite{Brezis1983}, we obtain the first part of this Proposition. The result concerning $\nabla^{\delta}u_{\delta}$ is proved in \cite{Chainais-Hillairet2003}.
\end{proof}


\subsection{Convergence of the scheme} 

Now it remains to prove that the function $u$ defined in Proposition \ref{compaciteu} satisfies Definition \ref{defisol} of a weak solution. The main difficulty in proving this comes from the fact that the diffusive and convective terms are put together in the Scharfetter-Gummel flux.

\begin{theorem}
Assume (H1)-(H7) hold. Then the function $u$ defined in Proposition \ref{compaciteu} satisfies the equation (\ref{eq})-(\ref{CI})-(\ref{BCD})-(\ref{BCN}) in the sense of (\ref{defsol}) and the boundary condition $u-\overline{u} \in L^{\infty}(0,T;H^{1}_{0}(\Omega))$.
\end{theorem}

\begin{proof}
Let $\psi \in \mathcal{D}(\Omega \times [0,T[)$ be a test function and $\psi^{n}_{K}=\psi(x_{K},t^{n})$ for all $K \in \mathcal{T}$ and $n \geq 0$. We suppose that $\delta >0$ is small enough such that Supp$(\psi) \subset \lbrace x \in \Omega; \text{ d}(x,\Gamma)>\delta \rbrace \times [0,(N_{T}-1)\Delta t[$. Let us define an approximate gradient of $\psi$ by
\begin{equation*}
\nabla^{\delta}\psi(x,t)=\left\{\begin{array}{lcl} \ds{\frac{\text{m}(\sigma)}{\text{m}(T_{K,\sigma})}\left(\psi^{n}_{L}-\psi^{n}_{K}\right) \mathbf{n}_{K,\sigma}} &  & \text{ if }  (x,t) \in T_{K,\sigma}\times [t^{n},t^{n+1}[,\,\ \sigma=K|L, \\ \ds{\frac{\text{m}(\sigma)}{\text{m}(T_{K,\sigma})}\left(\psi_{\sigma}^{n}-\psi^{n}_{K}\right) \mathbf{n}_{K,\sigma}} & & \text{ if }  (x,t) \in T_{K,\sigma}\times [t^{n},t^{n+1}[, \,\ \sigma \in \mathcal{E}_{ext,K}. \end{array}\right.
\end{equation*}
We get from \cite{Eymard2003} that $(\nabla^{\delta}\psi)_{\delta>0}$ weakly converges to $\nabla \psi$ in $(L^{2}(\Omega \times (0,T)))^{d}$ as $\delta$ goes to zero.\\
Let us introduce the following notations:
\begin{eqnarray*}
B_{10}(\delta) &=& -\left(\int_{0}^{T}\int_{\Omega}\ud(x,t)\,\pa_{t}\psi(x,t)\, dx \, dt+\int_{\Omega}\ud(x,0)\, \psi(x,0) \, dx\right),\\
B_{20}(\delta) &=& \int_{0}^{T}\int_{\Omega}r'(u_{\delta}(x,t-\Delta t))\, \nabla^{\delta}u_{\delta}(x,t)\cdot \nabla\psi(x,t)\, dx \, dt,\\
B_{30}(\delta) &=& -\int_{0}^{T}\int_{\Omega} \ud(x,t)\, \mathbf{q}(x) \cdot \nabla^{\delta}\psi(x,t)\, dx \, dt,
\end{eqnarray*}
and
\begin{equation*}
\varepsilon(\delta)=-B_{10}(\delta)-B_{20}(\delta)-B_{30}(\delta).
\end{equation*}
Multiplying the scheme (\ref{schemegene}) by $\Delta t \psi^{n}_{K}$ and summing through $K$ and $n$, we obtain
\begin{equation*}
B_{1}(\delta)+B_{2}(\delta)+B_{3}(\delta)=0,
\end{equation*}
where
\begin{eqnarray*}
B_{1}(\delta) &=& \sum_{n=0}^{N_{T}}\sum_{K \in \mathcal{T}}\text{m}(K)\left(U_{K}^{n+1}-U_{K}^{n}\right)\psi^{n}_{K},\\
B_{2}(\delta) &=& -\sum_{n=0}^{N_{T}}\Delta t \sum_{K \in \mathcal{T}} \sum_{\sigma \in \mathcal{E}_{K}}\frac{\text{m}(\sigma)q_{K,\sigma}}{2}\coth\left(\frac{d_{\sigma}q_{K,\sigma}}{2dr_{K,\sigma}^{n}}\right)\left(U^{n+1}_{\sigma}-U^{n+1}_{K}\right)\psi^{n}_{K},\\
B_{3}(\delta) &=& \sum_{n=0}^{N_{T}}\Delta t \sum_{K \in \mathcal{T}} \sum_{\sigma \in \mathcal{E}_{K}} \text{m}(\sigma)q_{K,\sigma}\frac{U_{K}^{n+1}+U_{\sigma}^{n+1}}{2}\psi^{n}_{K}.
\end{eqnarray*}
From the strong convergence of the sequence $(\ud)_{\delta>0}$ to $u$ in $L^{2}(\Omega \times ]0,T[)$, it is clear using the time translate estimate (\ref{transtime}) that there exists a subsequence of $(\ud)_{\delta>0}$, still denoted by   $(\ud)_{\delta>0}$, such that
\begin{equation*}
\ud( \,\cdot \, ,\, \cdot \,-\Delta t) \longrightarrow u \text{ in } L^{2}(\Omega \times ]0,T[) \text{ strongly as } \delta \rightarrow 0,
\end{equation*}
where $u \in L^{\infty}(0,T;H^{1}(\Omega))$ is defined in Proposition \ref{compaciteu}. Moreover, thanks to hypothesis (H4), we have $r' \in \mathcal{C}^{1}(\mathbb{R})$, and using the $L^{\infty}$-estimate (\ref{estimlinf}) we obtain that
\begin{equation*}
r'(\ud(\,\cdot \, ,\, \cdot \,-\Delta t))\longrightarrow r'(u) \text{ in } L^{2}(\Omega \times ]0,T[) \text{ strongly as } \delta \rightarrow 0.
\end{equation*}
Finally using this strong convergence and the weak convergence of the sequences $(\nabla^{\delta}u_{\delta})_{\delta>0}$ to $\nabla u$ and $(\nabla^{\delta}\psi)_{\delta>0}$ to $\nabla \psi$ in $(L^{2}(\Omega \times ]0,T[))^{d}$, it is easy to see that
\begin{eqnarray*}
\varepsilon(\delta)& \longrightarrow & \int_{0}^{T}\int_{\Omega}\left(u(x,t)\, \pa_{t}\psi-r'(u(x,t))\, \nabla u(x,t) \cdot \nabla \psi+u(x,t) \, \mathbf{q}(x) \cdot \nabla\psi\right) \, dx \, dt\\
& & +\int_{\Omega}u(x,0) \, \psi(x,0)\, dx, \text{ as } \delta \rightarrow 0.
\end{eqnarray*}
Therefore, it suffices to prove that $\varepsilon(\delta) \longrightarrow 0$ as $\delta \rightarrow 0$ and to this end we are going to prove that $\varepsilon(\delta)+B_{1}(\delta)+B_{2}(\delta)+B_{3}(\delta) \longrightarrow 0$ as $\delta \rightarrow 0$.\\

\textbf{Estimate of $B_{1}(\delta)-B_{10}(\delta)$.} This term is discussed for example in \cite{Chainais-Hillairet2003} (Theorem 5.2) and it is proved that:
\begin{equation*}
|B_{1}(\delta)-B_{10}(\delta)| \leq  \left[(T+1)\text{m}(\Omega)M\Vert \psi \Vert_{\mathcal{C}^{2}(\Omega \times (0,T))}\right]\delta \longrightarrow 0 \text{ as } \delta \rightarrow 0.
\end{equation*}

\textbf{Estimate of $B_{2}(\delta)-B_{20}(\delta)$.} Using a discrete integration by parts, we write
\begin{equation*}
B_{2}(\delta) = \sum_{n=0}^{N_{T}}\Delta t \sum_{\substack{\sigma \in \mathcal{E}_{int} \\ \sigma=K|L}}\frac{\text{m}(\sigma)q_{K,\sigma}}{2}\coth\left(\frac{d_{\sigma} q_{K,\sigma}}{2 dr^{n}_{K,\sigma}}\right)\left(U^{n+1}_{L}-U^{n+1}_{K}\right)(\psi_{L}^{n}-\psi_{K}^{n}).
\end{equation*}
Then we rewrite $B_{2}(\delta)=B_{21}(\delta)+B_{22}(\delta)+B_{23}(\delta)$, with
\begin{eqnarray*}
B_{21}(\delta) &=& \sum_{n=0}^{N_{T}}\Delta t \sum_{\substack{\sigma \in \mathcal{E}_{int} \\ \sigma=K|L}} \tau_{\sigma}r'(U^{n}_{K})\left(U^{n+1}_{L}-U^{n+1}_{K}\right)(\psi_{L}^{n}-\psi_{K}^{n}),\\
B_{22}(\delta) &=& \sum_{n=0}^{N_{T}}\Delta t \sum_{\substack{\sigma \in \mathcal{E}_{int} \\ \sigma=K|L}}\tau_{\sigma}\left(\frac{d_{\sigma}q_{K,\sigma}}{2dr^{n}_{K,\sigma}}\coth\left(\frac{d_{\sigma}q_{K,\sigma}}{2dr^{n}_{K,\sigma}}\right)-1\right)dr^{n}_{K,\sigma}\left(U^{n+1}_{L}-U^{n+1}_{K}\right)(\psi_{L}^{n}-\psi_{K}^{n}),\\
B_{23}(\delta) &=& \sum_{n=0}^{N_{T}}\Delta t \sum_{\substack{\sigma \in \mathcal{E}_{int} \\ \sigma=K|L}}\tau_{\sigma}\left(dr^{n}_{K,\sigma}-r'(U^{n}_{K})\right)\left(U^{n+1}_{L}-U^{n+1}_{K}\right)\left(\psi_{L}^{n}-\psi_{K}^{n}\right).
\end{eqnarray*}

Using the definition of $\tilde{u}_{\delta}$ and $\nabla^{\delta}u_{\delta}$, we rewrite $B_{20}(\delta)$ as $B_{210}(\delta)+B_{220}(\delta)$ with:
\begin{eqnarray*}
B_{210}(\delta) &=& \sum_{n=0}^{N_{T}}\sum_{\substack{\sigma \in \mathcal{E}_{int} \\ \sigma=K|L}}r'(U_{K}^{n})\frac{\text{m}(\sigma)}{\text{m}(T_{K,\sigma})}\left(U_{L}^{n+1}-U_{K}^{n+1}\right)\int_{t^{n}}^{t^{n+1}}\int_{T_{K,\sigma}}\nabla \psi(x,t) \cdot \mathbf{n}_{K,\sigma}\, dx \, dt,\\
B_{220}(\delta) &=& \sum_{n=0}^{N_{T}}\sum_{\substack{\sigma \in \mathcal{E}_{int} \\ \sigma=K|L}}\!\!\!\! \left(r'(U^{n}_{L})-r'(U^{n}_{K})\right)\frac{\text{m}(\sigma)}{\text{m}(T_{K,\sigma})}\left(U_{L}^{n+1}-U_{K}^{n+1}\right) \int_{t^{n}}^{t^{n+1}}\!\!\!\!\! \int_{T_{K,\sigma}\cap L}\!\!\!\!\!\!\!\!\!\!\! \nabla \psi(x,t) \cdot \mathbf{n}_{K,\sigma}\, dx \, dt.
\end{eqnarray*}
Now we prove that $B_{21}(\delta)-B_{210}(\delta) \rightarrow 0$ as $\delta \rightarrow 0$ and $B_{22}(\delta), \, B_{23}(\delta), \, B_{220}(\delta)\rightarrow 0$ as $\delta \rightarrow 0$.\\

\textbf{Estimate of $B_{21}(\delta)-B_{210}(\delta)$.} We have
\begin{equation*}
B_{21}(\delta)-B_{210}(\delta)=\sum_{n=0}^{N_{T}}\sum_{\substack{\sigma \in \mathcal{E}_{int}}} \!\!\!\! \text{m}(\sigma)r'(U^{n}_{K})\left[\int_{t^{n}}^{t^{n+1}}\!\!\!\! \left(\frac{\psi_{L}^{n}-\psi_{K}^{n}}{d_{\sigma}}-\frac{1}{\text{m}(T_{K,\sigma})}\int_{T_{K,\sigma}}\!\!\!\!\!\!\!\! \nabla\psi(x,t)\cdot \textbf{n}_{K,\sigma}\, dx\right)dt\right]\!\!.
\end{equation*}

Since the straight line $\overline{x_{K}x_{L}}$ is orthogonal to the edge $K|L$, we have $x_{L}-x_{K}=d_{\sigma}\mathbf{n}_{K,\sigma}$ and then from the regularity of $\psi$,
\begin{eqnarray*}
\frac{\psi_{L}^{n}-\psi_{K}^{n}}{d_{\sigma}}&=&\nabla \psi(x_{K},t^{n}) \cdot \mathbf{n}_{K,\sigma}+O(\Delta x)\\
&=& \nabla \psi(x,t) \cdot \mathbf{n}_{K,\sigma}+O(\delta), \,\ \forall (x,t) \in T_{K,\sigma} \times \left(t^{n},t^{n+1}\right).
\end{eqnarray*}
Then by taking the mean value over $T_{K,\sigma}$, there exists $D_{6}>0$ depending only on $\psi$ such that 
\begin{equation*}
\left\vert \int_{t^{n}}^{t^{n+1}}\left(\frac{\psi_{L}^{n}-\psi_{K}^{n}}{d_{\sigma}}-\frac{1}{\text{m}(T_{K,\sigma})}\int_{T_{K,\sigma}}\nabla \psi \cdot \mathbf{n}_{K,\sigma} \, dx \right)dt \right\vert \leq D_{6} \delta \Delta t,
\end{equation*}
and then 
\begin{equation*}
\left|B_{21}(\delta)-B_{210}(\delta)\right| \leq \delta D_{6} \sup_{s \in [m,M]}r'(s)\sum_{n=0}^{N_{T}}\Delta t\sum_{\substack{\sigma \in \mathcal{E}_{int}}}\text{m}(\sigma)\left|U_{L}^{n+1}-U_{K}^{n+1}\right|.
\end{equation*}
Since the straight line $\overline{x_{K}x_{L}}$ is orthogonal to the edge $\sigma=K|L$ for all $\sigma \in \mathcal{E}_{int,K}$ and the mesh is regular, there is a constant $D_{7}>0$ depending only on the dimension of the domain and the geometry of $\mathcal{T}$ such that  $\text{m}(\sigma)d_{\sigma} \leq D_{7} \text{m}(T_{K,\sigma})$ for all $K \in \mathcal{T}$ , all $\sigma \in \mathcal{E}_{ext,K}$ and then using the Cauchy-Schwarz inequality and the $L^{2}(0,T;H^{1})$ estimate (\ref{estiml2}), we obtain
\begin{equation*}
\left|B_{21}(\delta)-B_{210}(\delta)\right| \leq \delta D_{6} \sup_{s \in [m,M]}r'(s)\sqrt{D_{1}TD_{7}\text{m}(\Omega)} \longrightarrow 0 \text{ as } \delta \rightarrow 0.
\end{equation*}

\textbf{Estimate of $B_{22}(\delta)$.} Since $x \mapsto x\coth(x)$ is a 1-Lipschitz continuous function and is equal to 1 in 0, we have
\begin{eqnarray*}
\left|B_{22}(\delta)\right| &\leq & \sum_{n=0}^{N_{T}}\Delta t\sum_{\substack{\sigma \in \mathcal{E}_{int}}} \frac{\text{m}(\sigma)}{2}\left| q_{K,\sigma}\right| \left| U_{L}^{n+1}-U_{K}^{n+1}\right| \left| \psi_{L}^{n}-\psi_{K}^{n}\right|\\
& \leq & 2\delta \Vert \textbf{q} \Vert_{\infty}\sum_{n=0}^{N_{T}}\Delta t\sum_{\substack{\sigma \in \mathcal{E}_{int}}} \tau_{\sigma}\left| U_{L}^{n+1}-U_{K}^{n+1}\right| \left| \psi_{L}^{n}-\psi_{K}^{n}\right|, \text{ since $d_{\sigma} \leq 2 \delta$.}
\end{eqnarray*}
Then using the Cauchy-Schwarz inequality, the regularity of $\psi$ and the $L^{2}(0,T;H^{1})$ estimate (\ref{estiml2}), there exists $D_{8}>0$ only depending on $T$ and $\Omega$ such that:
\begin{equation*}
\left| B_{22}(\delta)\right| \leq \delta \Vert \textbf{q}\Vert_{\infty} D_{8} \Vert \psi \Vert_{\mathcal{C}^{1}} \sqrt{D_{1}} \longrightarrow 0 \text{ as } \delta \rightarrow 0.
\end{equation*}

\textbf{Estimate of $B_{23}(\delta)$.} Using Lemma \ref{propdr} and hypothesis (H4), we have
\begin{equation*}
\left| dr^{n}_{K,\sigma}-r'(U^{n}_{K})\right| \leq \sup_{s \in [m,M]}|r''(s)| \left|U_{L}^{n}-U^{n}_{K}\right|, \,\ \forall \sigma \in \mathcal{E}_{int}, \,\ \sigma=K|L.
\end{equation*}
Using the regularity of $\psi$ and the Cauchy-Schwarz inequality, we obtain
\begin{equation*}
\left| B_{23}(\delta) \right| \leq \delta \sup_{s \in [m,M]}|r''(s)| \Vert \psi \Vert_{\mathcal{C}^{1}} \sum_{n=0}^{N_{T}}\Delta t\sum_{\substack{\sigma \in \mathcal{E}_{int}}} \tau_{\sigma} \left|U_{L}^{n}-U^{n}_{K}\right| \left|U_{L}^{n+1}-U^{n+1}_{K}\right|,
\end{equation*}
and then using the $L^{2}(0,T;H^{1})$ estimate (\ref{estiml2}), we get
\begin{equation*}
\left| B_{23}(\delta) \right| \leq \delta \sup_{s \in [m,M]}|r''(s)| \Vert \psi \Vert_{\mathcal{C}^{1}} D_{1} \longrightarrow 0 \text{ as } \delta \rightarrow 0.
\end{equation*}

\textbf{Estimate of $B_{220}(\delta)$.} We obtain the same type of estimate as for $B_{23}(\delta)$:
\begin{equation*}
\left| B_{220}(\delta) \right| \leq 2\delta \sup_{s \in [m,M]}|r''(s)| \Vert \psi \Vert_{\mathcal{C}^{1}} D_{1} \longrightarrow 0 \text{ as } \delta \rightarrow 0.
\end{equation*}

\textbf{Estimate of $B_{3}(\delta)-B_{30}(\delta)$.} Using a discrete integration by parts, we obtain
\begin{equation*}
B_{3}(\delta) = -\sum_{n=0}^{N_{T}}\Delta t\sum_{\substack{\sigma \in \mathcal{E}_{int}}} \text{m}(\sigma)q_{K,\sigma}\frac{U_{K}^{n+1}+U_{L}^{n+1}}{2}\left(\psi_{L}^{n}-\psi_{K}^{n}\right),
\end{equation*}
and then we rewrite $B_{3}(\delta)$ as $B_{31}(\delta)+B_{32}(\delta)$, with
\begin{eqnarray*}
B_{31}(\delta)&=&-\sum_{n=0}^{N_{T}}\Delta t\sum_{\substack{\sigma \in \mathcal{E}_{int}}} \text{m}(\sigma)q_{K,\sigma}\frac{U_{L}^{n+1}-U_{K}^{n+1}}{2}\left(\psi_{L}^{n}-\psi_{K}^{n}\right),\\
B_{32}(\delta)&=& -\sum_{n=0}^{N_{T}}\Delta t\sum_{\substack{\sigma \in \mathcal{E}_{int}}} \text{m}(\sigma)q_{K,\sigma}U^{n+1}_{K}\left(\psi_{L}^{n}-\psi_{K}^{n}\right).
\end{eqnarray*}
Using the definition of $\nabla^{\delta}\psi$, we get
\begin{equation*}
B_{30}(\delta)=-\sum_{n=0}^{N_{T}}\sum_{\substack{\sigma \in \mathcal{E}_{int}}}\int_{t^{n}}^{t^{n+1}}\int_{T_{K,\sigma}}u_{\delta}(x,t)\frac{\text{m}(\sigma)}{\text{m}(T_{K,\sigma})}\left(\psi_{L}^{n}-\psi_{K}^{n}\right)\textbf{q}(x)\cdot \textbf{n}_{K,\sigma}\, dx \, dt,
\end{equation*}
which gives, using the definition of $u_{\delta}$, $B_{30}(\delta)=B_{310}(\delta)+B_{320}(\delta)$, where
\begin{eqnarray*}
B_{310}(\delta) &=& -\sum_{n=0}^{N_{T}}\Delta t\sum_{\substack{\sigma \in \mathcal{E}_{int}}}\text{m}(\sigma)\left(U^{n+1}_{L}-U^{n+1}_{K}\right)\left(\psi_{L}^{n}-\psi^{n}_{K}\right)\frac{1}{\text{m}(T_{K,\sigma})}\int_{T_{K,\sigma}\cap L}\textbf{q}(x)\cdot \textbf{n}_{K,\sigma}\, dx,\\
B_{320}(\delta) &= & -\sum_{n=0}^{N_{T}}\sum_{\substack{\sigma \in \mathcal{E}_{int}}}\text{m}(\sigma)U_{K}^{n+1}\left(\psi_{L}^{n}-\psi^{n}_{K}\right)\frac{1}{\text{m}(T_{K,\sigma})}\int_{T_{K,\sigma}}\textbf{q}(x)\cdot \textbf{n}_{K,\sigma}\, dx.
\end{eqnarray*}
Now we prove that $B_{32}(\delta)-B_{320}(\delta) \rightarrow 0$ as $\delta \rightarrow 0$ and $B_{31}(\delta), \, B_{310}(\delta)\rightarrow 0$ as $\delta \rightarrow 0$.\\
Using the regularity of $\textbf{q}$, there exists $D_{9}>0$ which does not depend on $\delta$ such that 
\begin{equation*}
\left| \frac{1}{\text{m}(\sigma)}\int_{\sigma}\textbf{q}(x)\cdot \textbf{n}_{K,\sigma} \, ds(x)-\frac{1}{\text{m}(T_{K,\sigma})}\int_{T_{K,\sigma}}\textbf{q}(x)\cdot \textbf{n}_{K,\sigma} \, dx\right| \leq D_{9} \delta.
\end{equation*}
Then we can estimate $B_{32}(\delta)-B_{320}(\delta)$:
\begin{eqnarray*}
\left| B_{32}(\delta)-B_{320}(\delta)\right| &\leq & \delta D_{9} M \sum_{n=0}^{N_{T}}\Delta t\sum_{\substack{\sigma \in \mathcal{E}_{int}}}\text{m}(\sigma) \left|\psi_{L}^{n}-\psi_{K}^{n}\right|\\
& \leq & \delta D_{8} D_{9} M \Vert \psi \Vert_{\mathcal{C}^{1}}\sqrt{D_{7}\text{m}(\Omega)} \longrightarrow 0 \text{ as } \delta \rightarrow 0.
\end{eqnarray*}
Moreover, we have
\begin{eqnarray*}
\left|B_{31}(\delta)\right| & \leq & \delta \Vert \textbf{q} \Vert_{\infty} \sum_{n=0}^{N_{T}}\Delta t\sum_{\substack{\sigma \in \mathcal{E}_{int}}}\tau_{\sigma}\left|U_{L}^{n+1}-U_{K}^{n+1}\right| \left|\psi_{L}^{n}-\psi_{K}^{n}\right|\\
& \leq & \delta \Vert \textbf{q} \Vert_{\infty} \Vert \psi \Vert_{\mathcal{C}^{1}}D_{8}\sqrt{D_{1}} \longrightarrow 0 \text{ as } \delta \rightarrow 0.
\end{eqnarray*}
We obtain in the same way that $B_{310}(\delta) \longrightarrow 0$ as $\delta \rightarrow 0$.\\

Hence $u$ satisfies
\begin{eqnarray*}
& & \int_{0}^{T}\int_{\Omega}\left(u(x,t)\, \pa_{t}\psi(x,t)+r'(u(x,t))\, \nabla u(x,t)\cdot \nabla\psi(x,t)+u(x,t)\, \mathbf{q}(x)\cdot \nabla\psi(x,t)\right)\, dx \, dt\\
& &+\int_{\Omega}u(x,0)\, \psi(x,0)\, dx=0,
\end{eqnarray*}
and then 
\begin{eqnarray*}
& & \int_{0}^{T}\int_{\Omega}\left(u(x,t)\, \pa_{t}\psi(x,t)+\nabla(r(u(x,t)))\cdot \nabla\psi(x,t)+u(x,t)\, \mathbf{q}(x)\cdot \nabla\psi(x,t)\right)\, dx \, dt\\
& &+\int_{\Omega}u(x,0)\, \psi(x,0)\, dx=0.
\end{eqnarray*}
It remains to show that $u-\overline{u} \in L^{\infty}(0,T;H^{1}_{0}(\Omega))$. This proof is based on the $L^{2}(0,T;H^{1}(\Omega))$ estimate (\ref{estiml2}) and is similar to the one of Theorem 5.1 in \cite{Chainais-Hillairet2003}.

\end{proof}


\section{Numerical simulations}

\subsection{Order of convergence}

We consider the following one dimensional test case, picked in the paper of R. Eymard, J. Fuhrmann and K. Gärtner \cite{Eymard2006}. We look at the case where, in (\ref{eq}) we take $\Omega=(0,1)$, $T=0.004$, $r: s \mapsto s^{2}$, $q=100$, in (\ref{CI}) we take $u_{0}=0$ and in (\ref{BCD}) we take, for $v=200$,
\begin{eqnarray*}
\overline{u}(0,t) &=& (v-q)vt/2\\
\overline{u}(1,t) &=& \left\{ \begin{array}{lcl} 0 & & \text{ for } t<1/v, \\ (v-q)(vt-1)/2 & & \text{ otherwise.}\end{array}\right.
\end{eqnarray*}
The unique weak solution of this problem is then given by
\begin{equation*}
u(x,t)=\left\{\begin{array}{lcl} (v-q)(vt-x)/2 & & \text{ if } x<vt, \\ 0 & & \text{ if } x \geq vt.\end{array}\right.
\end{equation*}
The time step is taken equal to $ \Delta t=10^{-8}$ to study the order of convergence with respect to the spatial step size $\Delta x$. In Tables \ref{tableconvergenceinf} and \ref{tableconvergence2}, we compare the order of convergence in $L^{\infty}$ and $L^{2}$ norms of the scheme (\ref{CId})-(\ref{BCd})-(\ref{schemegene}) defined on one hand with the classical upwind flux (\ref{upwindclass}) and on the other hand with the Scharfetter-Gummel extended flux (\ref{flux}). We obtain the same order of convergence as in \cite{Eymard2006}. Moreover, it appears that even if we are in a degenerate case, the Scharfetter-Gummel extended scheme is more accurate than the classical upwind scheme.

\begin{table}[!ht]
\centering
\begin{tabular}{|c|l|c|c|c|c|}
\hline $j$ & $\Delta x(j)$ & $\Vert u-u_{\delta} \Vert_{L^{\infty}}$  & Order & $\Vert u-u_{\delta} \Vert_{L^{\infty}}$ & Order \\ 
  &  & Upwind &  & SG extended &  \\ 
\hline 0 & $ 2.5.10^{-2}  $   & $ 1.110         $  &      & $ 2.137.10^{-1} $  &   \\ 
       1 & $ 1.25.10^{-2} $   & $ 7.237.10^{-1} $  & 0.62 & $ 1.107.10^{-1} $  & 0.95  \\ 
       2 & $ 6.3.10^{-3}  $   & $ 4.485.10^{-1} $  & 0.69 & $ 5.631.10^{-2} $  & 0.98 \\ 
       3 & $ 3.1.10^{-3}  $   & $ 2.685.10^{-1} $  & 0.74 & $ 2.84.10^{-2}  $  & 0.99  \\ 
       4 & $ 1.6.10^{-3}  $   & $ 1.568.10^{-1} $  & 0.78 & $ 1.426.10^{-2} $  & 1  \\
       5 & $ 8.10^{-4}    $   & $ 9.10^{-2}     $  & 0.80 & $ 7.15.10^{-3}  $  & 1  \\
\hline
\end{tabular} 
\caption{Experimental order of convergence in $L^{\infty}$ norm for spatial step sizes $ \Delta x(j)=\ds{\frac{0.1}{2^{j+2}}}$ of the classical upwind scheme and of the Scharfetter-Gummel extended scheme.}
\label{tableconvergenceinf}
\end{table}

\begin{table}[!ht]
\centering
\begin{tabular}{|c|l|c|c|c|c|}
\hline $j$ & $\Delta x(j)$ & $\Vert u-u_{\delta} \Vert_{L^{2}}$  & Order & $\Vert u-u_{\delta} \Vert_{L^{2}}$ & Order \\ 
  &  & Upwind &  & SG extended &  \\ 
\hline 0 & $ 2.5.10^{-2}  $   & $ 3.336.10^{-1} $  &      & $ 4.806.10^{-2} $  &   \\ 
       1 & $ 1.25.10^{-2} $   & $ 1.852.10^{-1} $  & 0.85 & $ 1.642.10^{-2} $  & 1.55  \\ 
       2 & $ 6.3.10^{-3}  $   & $ 9.911.10^{-2} $  & 0.9  & $ 5.695.10^{-3} $  & 1.53 \\ 
       3 & $ 3.1.10^{-3}  $   & $ 5.182.10^{-2} $  & 0.94 & $ 2.10^{-3}     $  & 1.51  \\ 
       4 & $ 1.6.10^{-3}  $   & $ 2.669.10^{-2} $  & 0.96 & $ 7.142.10^{-4} $  & 1.49  \\
       5 & $ 8.10^{-4}    $   & $ 1.361.10^{-2} $  & 0.97 & $ 2.695.10^{-4} $  & 1.41  \\
\hline
\end{tabular} 
\caption{Experimental order of convergence in $L^{2}$ norm for spatial step sizes $ \Delta x(j)=\ds{\frac{0.1}{2^{j+2}}}$ of the classical upwind scheme and of the Scharfetter-Gummel extended scheme.}
\label{tableconvergence2}
\end{table}


\subsection{Large time behavior}


\subsubsection{The drift-diffusion system for semiconductors}

We may define the finite volume approximation of the drift-diffusion system (\ref{DD}). Initial and boundary conditions are approximated by (\ref{CId}) and (\ref{BCd}). The doping profile is approximated by $(C_{K})_{K \in \mathcal{T}}$ by taking the mean value of $C$ on each volume $K$. The scheme for the system (\ref{DD}) is given by:
\begin{equation*}
\left\{\begin{array}{lcl}\text{m}(K)\ds{\frac{N^{n+1}_{K}-N^{n}_{K}}{\Delta t}}+\sum_{\sigma \in \mathcal{E}_{K}} \mathcal{F}_{K,\sigma}^{n+1}=0, & & \forall K \in \mathcal{T}, \forall n \geq 0,\\ 
\text{m}(K)\ds{\frac{P^{n+1}_{K}-P^{n}_{K}}{\Delta t}}+\sum_{\sigma \in \mathcal{E}_{K}} \mathcal{G}_{K,\sigma}^{n+1}=0, & & \forall K \in \mathcal{T}, \forall n \geq 0,\\ 
\sum_{\sigma \in \mathcal{E}_{K}}\tau_{\sigma}DV_{K,\sigma}^{n}=\text{m}(K)\left(N_{K}^{n}-P_{K}^{n}-C_{K}\right), & & \forall K \in \mathcal{T}, \forall n \geq 0, \end{array}\right.
\end{equation*}
where 
\begin{equation*}
\mathcal{F}_{K,\sigma}^{n+1}= \tau_{\sigma} dr\left(N_{K}^{n},N_{\sigma}^{n}\right) \left(B\left(\ds{\frac{-DV_{K,\sigma}^{n}}{dr\left(N_{K}^{n},N_{\sigma}^{n}\right)}}\right)N_{K}^{n+1}-B\left(\ds{\frac{DV_{K,\sigma}^{n}}{dr(N_{K}^{n},N_{\sigma}^{n})}}\right)N_{\sigma}^{n+1}\right), \,\ \forall \sigma \in \mathcal{E}_{K},
\end{equation*}
and
\begin{equation*}
\mathcal{G}_{K,\sigma}^{n+1}= \tau_{\sigma} dr(P_{K}^{n},P_{\sigma}^{n}) \left(B\left(\ds{\frac{DV_{K,\sigma}^{n}}{dr(P_{K}^{n},P_{\sigma}^{n})}}\right)P_{K}^{n+1}-B\left(\ds{\frac{-DV_{K,\sigma}^{n}}{dr(P_{K}^{n},P_{\sigma}^{n})}}\right)P_{\sigma}^{n+1}\right), \,\ \forall \sigma \in \mathcal{E}_{K}.
\end{equation*}
We compute an approximation $(N^{eq}_{K},P^{eq}_{K},V^{eq}_{K})_{K \in \mathcal{T}}$ of the thermal equilibrium $(N^{eq},P^{eq},V^{eq})$ defined by (\ref{eqNP})-(\ref{eqV}) with the finite volume scheme proposed by C. Chainais-Hillairet and F. Filbet in \cite{Chainais-Hillairet2007}.\\
Then we introduce the discrete version of the deviation of the total energy from the thermal equilibrium (\ref{Econtinusc}): for $n\geq 0$,
\begin{eqnarray*}
\mathcal{E}^{n} &=& \sum_{K \in \mathcal{T}}\text{m}(K) \left(H(N^{n}_{K})-H(N^{eq}_{K})-h(N^{eq}_{K})\left(N^{n}_{K}-N^{eq}_{K}\right)\right)\\
& & + \sum_{K \in \mathcal{T}}\text{m}(K) \left(H(P^{n}_{K})-H(P^{eq}_{K})-h(P^{eq}_{K})(P^{n}_{K}-P^{eq}_{K})\right)\\
& & + \frac{1}{2}\sum_{\substack{\sigma \in \mathcal{E}_{int} \\ \sigma=K|L}}\tau_{\sigma}\left\vert DV_{K,\sigma}^{n}-DV_{K,\sigma}^{eq} \right\vert^{2}+\frac{1}{2}\sum_{K\in \mathcal{T}}\sum_{\sigma \in \mathcal{E}_{ext,K}^{D}}\tau_{\sigma}\left\vert DV_{K,\sigma}^{n}-DV_{K,\sigma}^{eq} \right\vert^{2},
\end{eqnarray*}
and the discrete version of the energy dissipation (\ref{Icontinusc}): for $n \geq 0$,
\begin{eqnarray*}
\mathcal{I}^{n} &=& \sum_{\substack{\sigma \in \mathcal{E}_{int} \\ \sigma=K|L}}\tau_{\sigma}\min\left(N_{K}^{n+1},N_{L}^{n+1}\right)\left[D\left(h\left(N^{n+1}\right)-V^{n}\right)_{K,\sigma}\right]^{2}\\
& & + \sum_{K \in \mathcal{T}}\sum_{\sigma \in \mathcal{E}_{ext,K}}\tau_{\sigma}\min\left(N_{K}^{n+1},N_{\sigma}^{n+1}\right)\left[D\left(h\left(N^{n+1}\right)-V^{n}\right)_{K,\sigma}\right]^{2}\\
& & + \sum_{\substack{\sigma \in \mathcal{E}_{int} \\ \sigma=K|L}}\tau_{\sigma}\min\left(P_{K}^{n+1},P_{L}^{n+1}\right)\left[D\left(h\left(P^{n+1}\right)+V^{n}\right)_{K,\sigma}\right]^{2}\\
& & + \sum_{K \in \mathcal{T}}\sum_{\sigma \in \mathcal{E}_{ext,K}}\tau_{\sigma}\min\left(P_{K}^{n+1},P_{\sigma}^{n+1}\right)\left[D\left(h\left(P^{n+1}\right)+V^{n}\right)_{K,\sigma}\right]^{2}.
\end{eqnarray*}
We present a test case for a geometry corresponding to a PN-junction in 2D  picked in the paper of C. Chainais-Hillairet and F. Filbet \cite{Chainais-Hillairet2007}. The doping profile is piecewise constant, equal to $+1$ in the N-region and $-1$ in the P-region.\\
The Dirichlet boundary conditions are
\begin{eqnarray*}
\overline{N}=0.1, \,\ \overline{P}=0.9, \,\ \overline{V}=\frac{h(\overline{N})-h(\overline{P})}{2} & & \text{ on } \{y=1, \,\ 0 \leq x \leq 0.25\},\\
\overline{N}=0.9, \,\ \overline{P}=0.1, \,\ \overline{V}=\frac{h(\overline{N})-h(\overline{P})}{2} & & \text{ on } \{y=0\}.
\end{eqnarray*}
Elsewhere, we put homogeneous Neumann boundary conditions.\\
The pressure is nonlinear: $r(s)=s^{\gamma}$ with $\gamma=5/3$, which corresponds to the isentropic model.\\
We compute the numerical approximation of the thermal equilibrium and of the transient drift-diffusion system on a mesh made of 896 triangles, with time step $\Delta t=0.01$.\\
We then compare the large time behavior of approximate solutions obtained with the three following fluxes:
\begin{itemize}
\item the upwind flux defined by (\ref{upwindclass}) (\textbf{Upwind}),
\item the Scharfetter-Gummel extended flux (\ref{flux}) with the first choice (\ref{dr1}) of $dr_{K,\sigma}$, close to that of Jüngel and Pietra (\textbf{SG-JP}),
\item the Scharfetter-Gummel extended flux (\ref{flux}) with the new definition (\ref{dr2}) of $dr_{K,\sigma}$ (\textbf{SG-ext}).
\end{itemize}
In Figure \ref{EetIsc} we compare the discrete relative energy $\mathcal{E}^{n}$ and its dissipation $\mathcal{I}^{n}$ obtained with the \textbf{Upwind} flux, the \textbf{SG-JP} flux and the \textbf{SG-ext} flux. With the third scheme, we observe that $\mathcal{E}^{n}$ and $\mathcal{I}^{n}$ converge to zero when time goes to infinity, without a saturation phenomenon. This scheme is the only one of the three which preserves thermal equilibrium, so it appears that this property is crucial to have a good asymptotic behavior.\\
In Figure \ref{EetIdt} we compare the relative energy $\mathcal{E}^{n}$ and its dissipation $\mathcal{I}^{n}$ obtained with the \textbf{SG-ext} flux for three different time steps $\Delta t=5.10^{-3}, \, 10^{-3}, \, 10^{-4}$. It appears that the decay rate does not depend on the time step.

\begin{figure}
\centering
\subfigure{\includegraphics[scale=0.45]{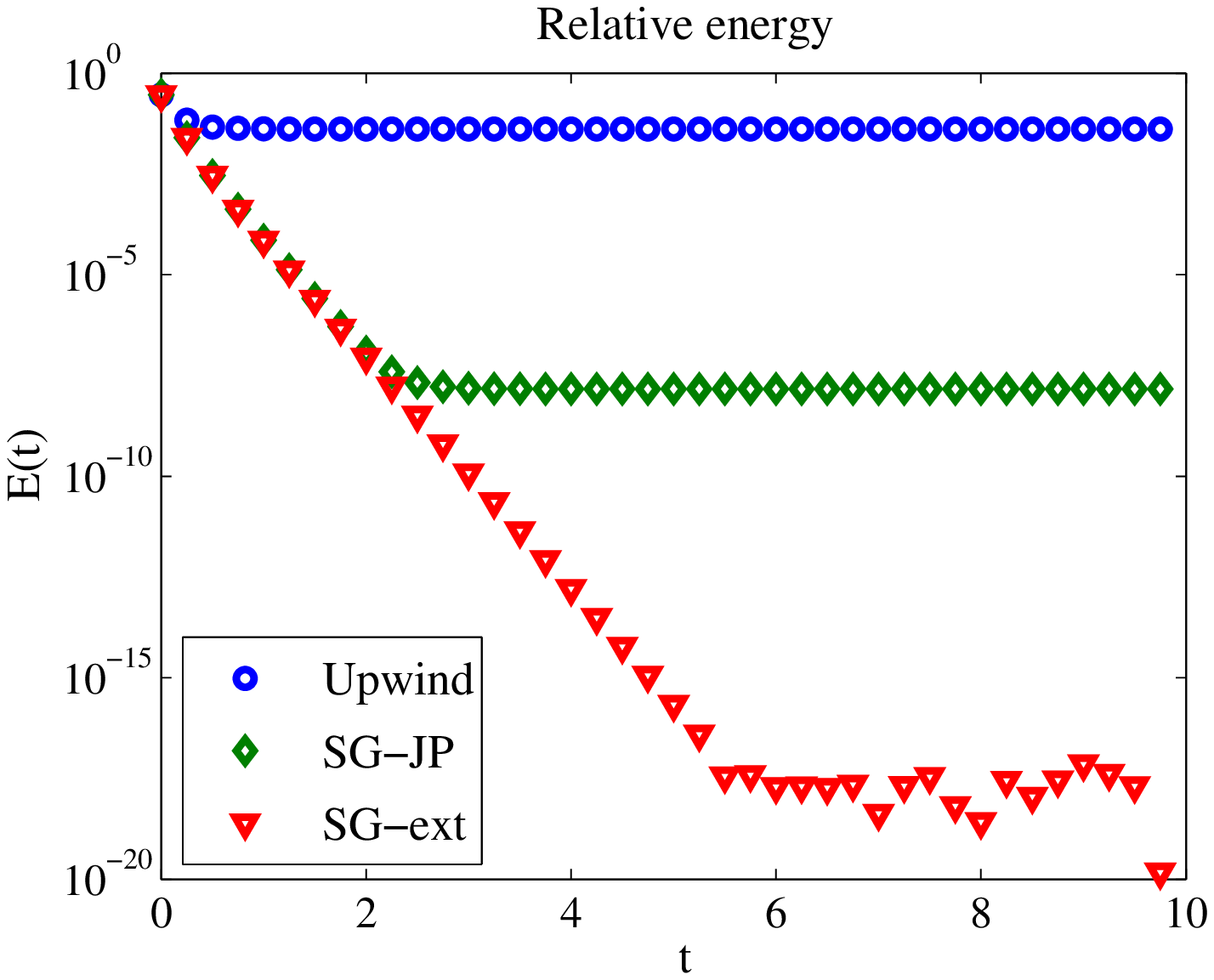}}
\subfigure{\includegraphics[scale=0.45]{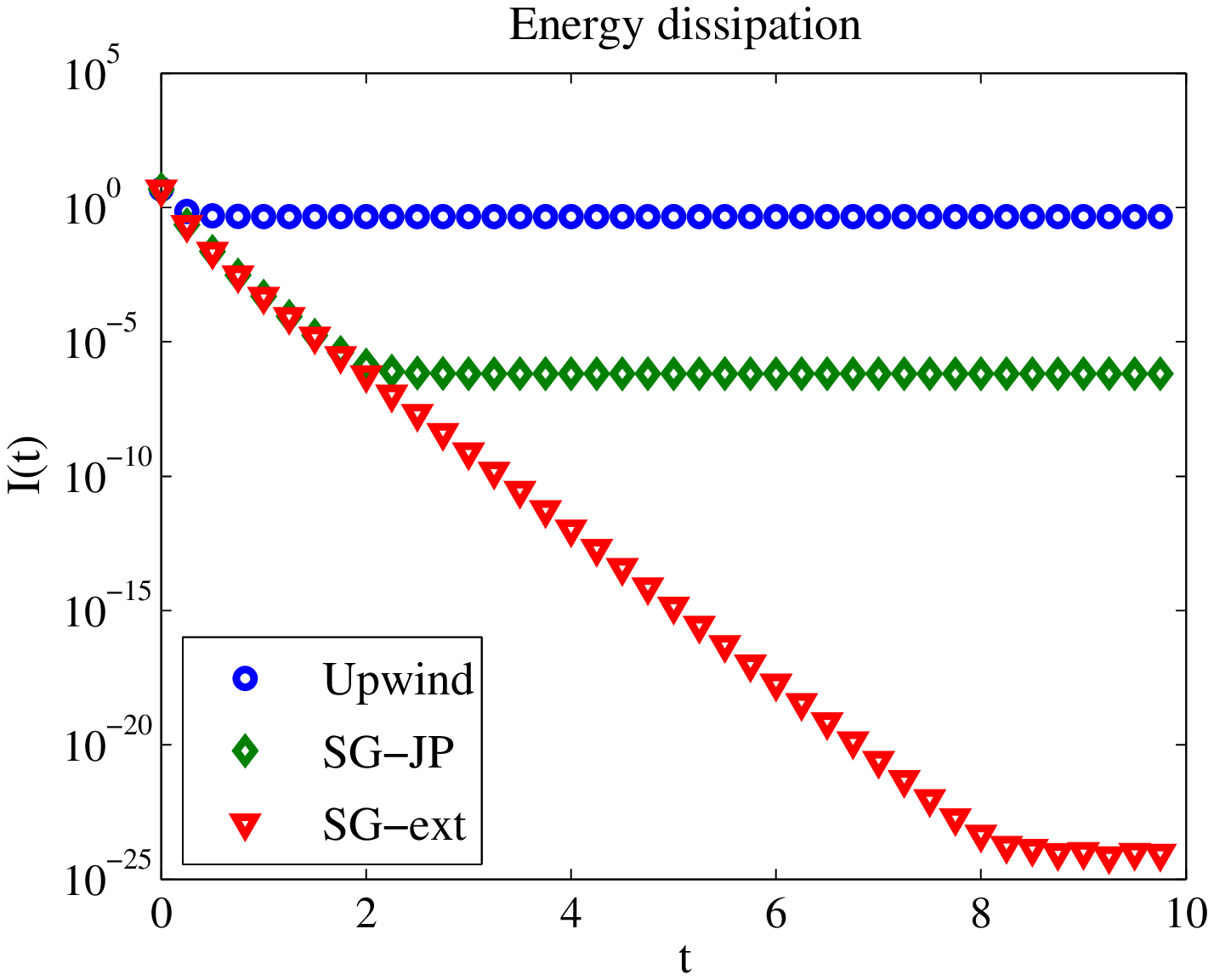}}
\caption{Evolution of the relative energy $\mathcal{E}^{n}$ and its dissipation $\mathcal{I}^{n}$ in log-scale for different schemes.}
\label{EetIsc}
\end{figure}

\begin{figure}
\centering
\subfigure{\includegraphics[scale=0.45]{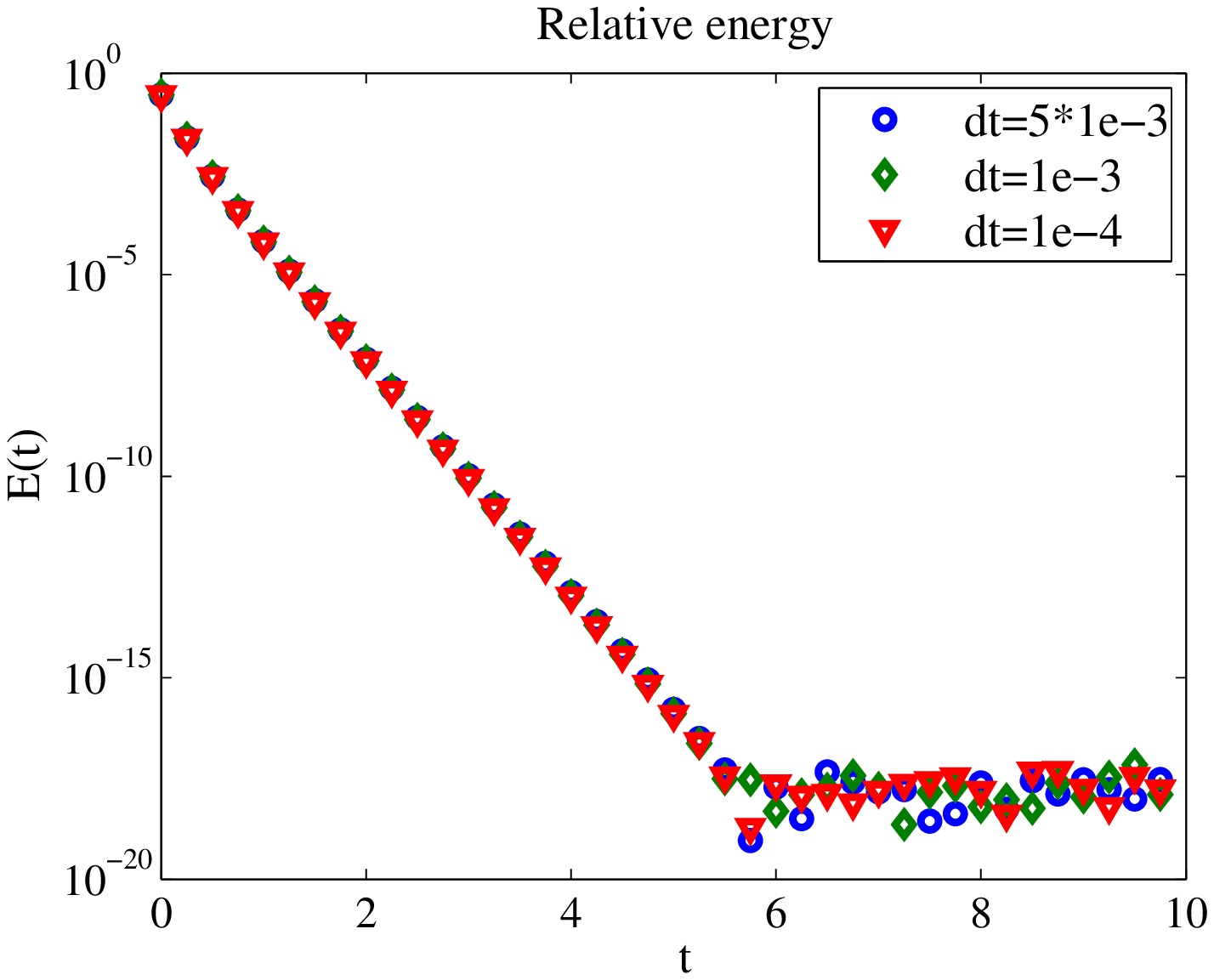}}
\subfigure{\includegraphics[scale=0.45]{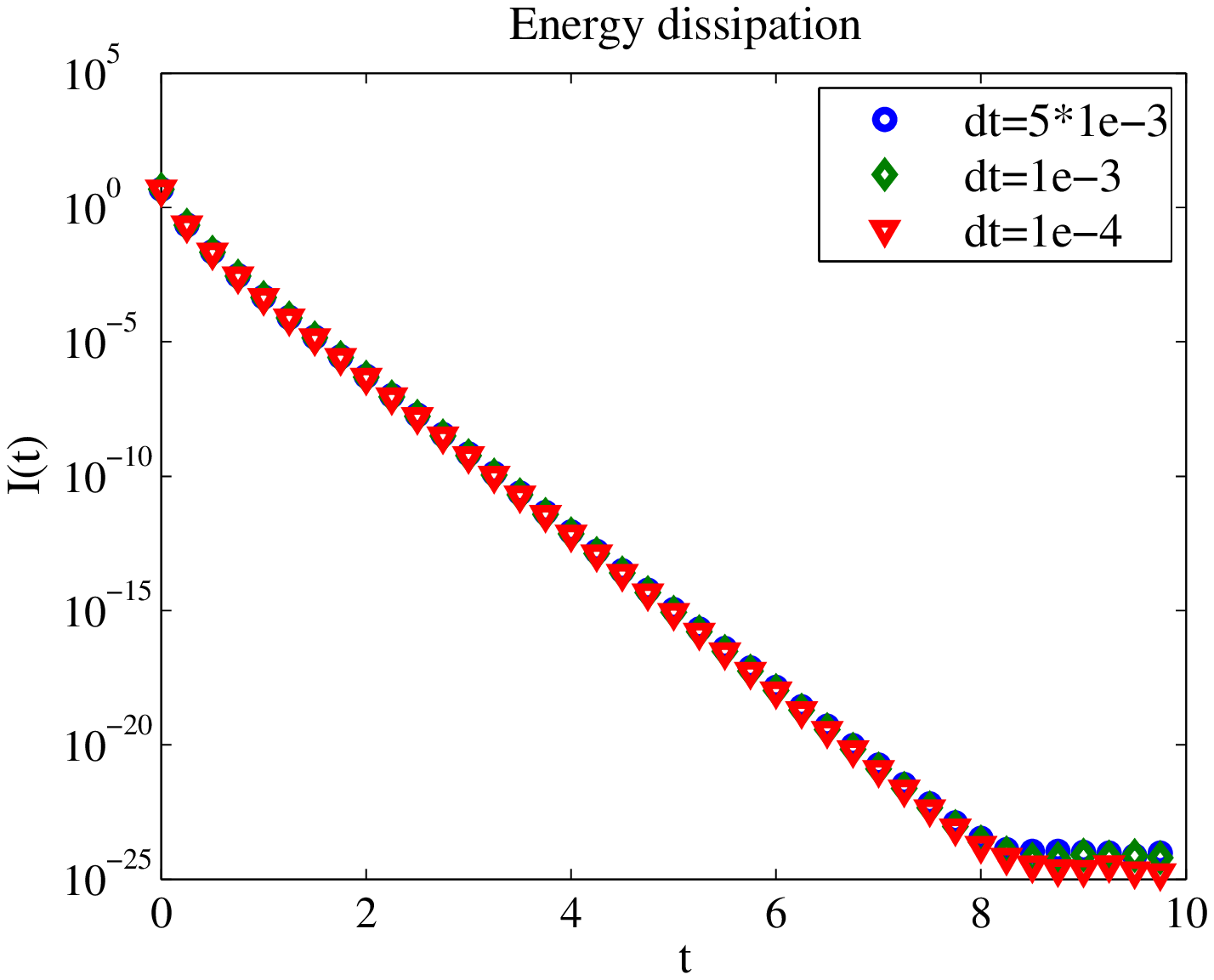}}
\caption{The relative energy $\mathcal{E}^{n}$ and its dissipation $\mathcal{I}^{n}$ in log-scale for different time steps.}
\label{EetIdt}
\end{figure}


\subsubsection{The porous media equation}

We recall that the unique stationary solution $u^{eq}$ of the porous media equation (\ref{PM}) is given by the Barenblatt-Pattle type formula (\ref{barenblatt}), where $C_{1}$ is such that $u^{eq}$ as the same mass as the initial data $u_{0}$. We define an approximation $\left(U^{eq}_{K}\right)_{K \in \mathcal{T}}$ of $u^{eq}$ by
\begin{equation*}
U^{eq}_{K}=\left(\tilde{C}_{1}-\frac{\gamma -1}{2\gamma}\left\vert x_{K}\right\vert^{2}\right)^{1/(\gamma -1)}_{+}, \,\ K \in \mathcal{T},
\end{equation*}
where $\tilde{C}_{1}$ is such that the discrete mass of $\left(U^{eq}_{K}\right)_{K \in \mathcal{T}}$ is equal to that of $\left(U^{0}_{K}\right)_{K \in \mathcal{T}}$, namely $\ds{\sum_{K\in \mathcal{T}}\text{m}(K)U^{eq}_{K}=\sum_{K \in \mathcal{T}}\text{m}(K)U^{0}_{K}}$. We use a fixed point algorithm to compute this constant $ \tilde{C}_{1}$.\\
We introduce the discrete version of the relative entropy (\ref{Econtinupm})
\begin{equation*}
\mathcal{E}^{n}=\sum_{K \in \mathcal{T}}\text{m}(K)\left(H(U_{K}^{n})-H(U_{K}^{eq})+\frac{|x_{K}|^{2}}{2}(U_{K}^{n}-U_{K}^{eq})\right),
\end{equation*}
and the discrete version of the entropy dissipation (\ref{Icontinupm})
\begin{eqnarray*}
\mathcal{I}^{n} &=& \sum_{\substack{\sigma \in \mathcal{E}_{int}\\ \sigma=K|L}} \tau_{\sigma}\min\left(U_{K}^{n},U_{L}^{n}\right)\left\vert D\left(h(U^{n})+\frac{|x|^{2}}{2}\right)_{K,\sigma}\right\vert^{2}\\
& &+ \sum_{K \in \mathcal{T}}\sum_{\sigma \in \mathcal{E}_{ext,K}}\tau_{\sigma}\min\left(U_{K}^{n},U_{\sigma}^{n}\right)\left\vert D\left(h(U^{n})+\frac{|x|^{2}}{2}\right)_{K,\sigma}\right\vert^{2}.
\end{eqnarray*}
We consider the following two dimensional test case: $r(s)=s^{3}$, with initial condition
\begin{equation*}
u_{0}(x,y)=\left\{\begin{array}{ccl} \exp\left(-\frac{1}{6-(x-2)^{2}-(y+2)^{2}}\right) &\text{ if } & (x-2)^{2}+(y+2)^{2}<6, \\ \exp\left(-\frac{1}{6-(x+2)^{2}-(y-2)^{2}}\right) &\text{ if } & (x+2)^{2}+(y-2)^{2}<6, \\ 0 & \text{ otherwise, } & \end{array}\right.
\end{equation*}
and periodic boundary conditions. \\
Then we compute the approximate solution on $\Omega \times (0,10)$ with $\Omega=(-10,10) \times (-10,10)$. We consider a uniform cartesian grid with $100 \times 100$ points and the time step is fixed to $\Delta t=5.10^{-4}$.\\
In Figure \ref{evolu}, we plot the evolution of the numerical solution $u$ computed with the \textbf{SG-ext} flux at three different times $t=0$, $t=0.4$ and $t=4$ and the approximation of the Barenblatt-Pattle solution. In Figure \ref{EetImp} we compare the relative entropy $\mathcal{E}^{n}$ and its dissipation $\mathcal{I}^{n}$ computed with the scheme (\ref{schemegene}) and different fluxes: the \textbf{Upwind} flux, the \textbf{SG-JP} flux and the \textbf{SG-ext} flux. We made the same findings as in the case of the drift-diffusion system for semiconductors: the third scheme is the only one of the three for which there is no saturation phenomenon, which confirms the importance of preserving the equilibrium to obtain a consistent asymptotic behavior of the approximate solution. Moreover it appears that the entropy decays exponentially fast, which has been proved in \cite{Carrillo2000}.\\
In Figure \ref{norml1mp}, we represent the discrete $L^{1}$ norm of $U-U^{eq}$ (obtained with the \textbf{SG-ext} flux) in log scale. According to the paper of J. A. Carrillo and G. Toscani, there exists a constant $C>0$ such that, in this case,  
\begin{equation*}
\Vert u(t,x)-u^{eq}(x)\Vert_{L^{1}(\mathbb{R})} \leq C \exp\left(-\frac{3}{5}t\right), \,\ t \geq 0.
\end{equation*}
We observe that the experimental decay of $u$ towards the steady state $u^{eq}$ is exponential, at a rate better than $\ds{\frac{3}{5}}$. 

\begin{figure}[!ht]
\centering
\subfigure{\includegraphics[scale=0.55]{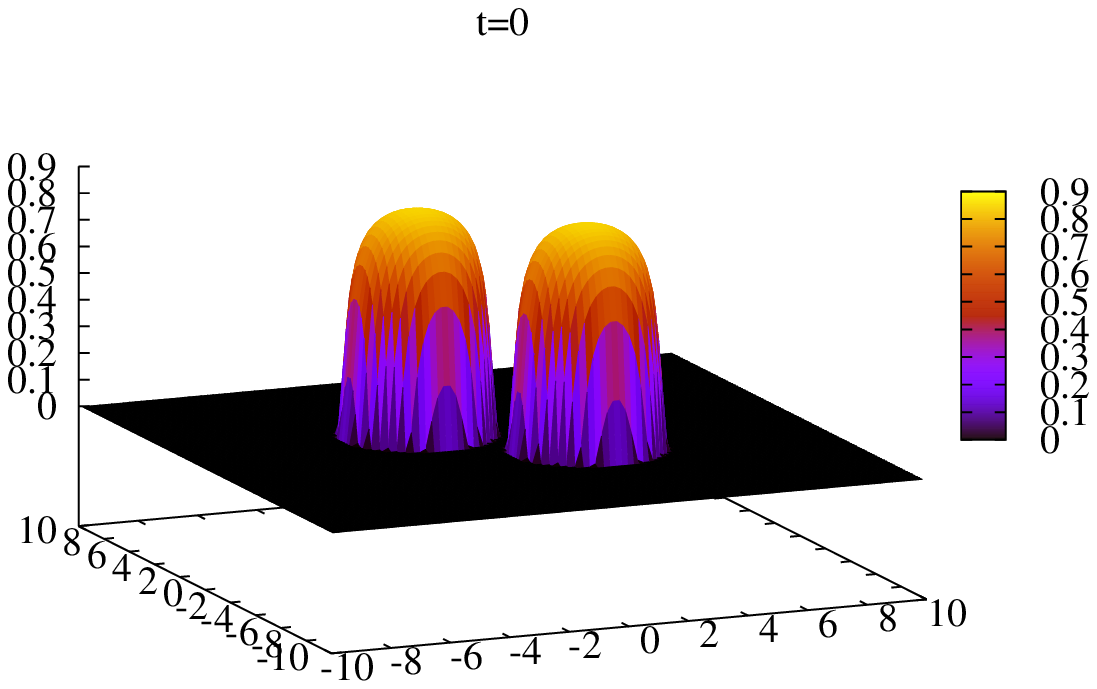}}
\subfigure{\includegraphics[scale=0.55]{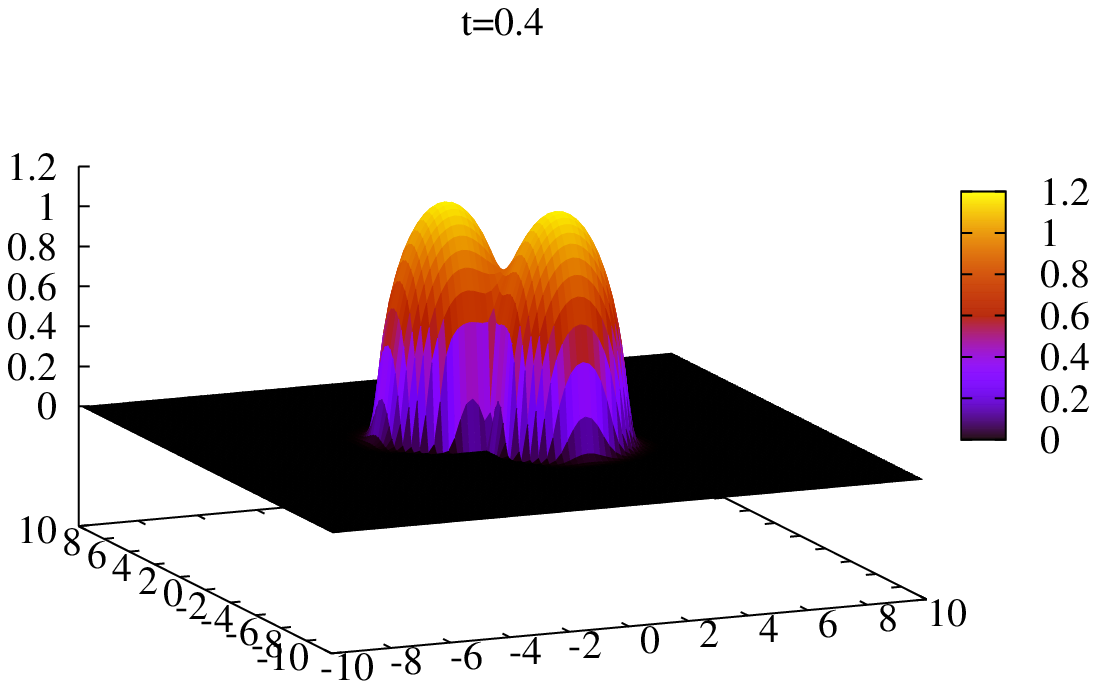}}
\subfigure{\includegraphics[scale=0.55]{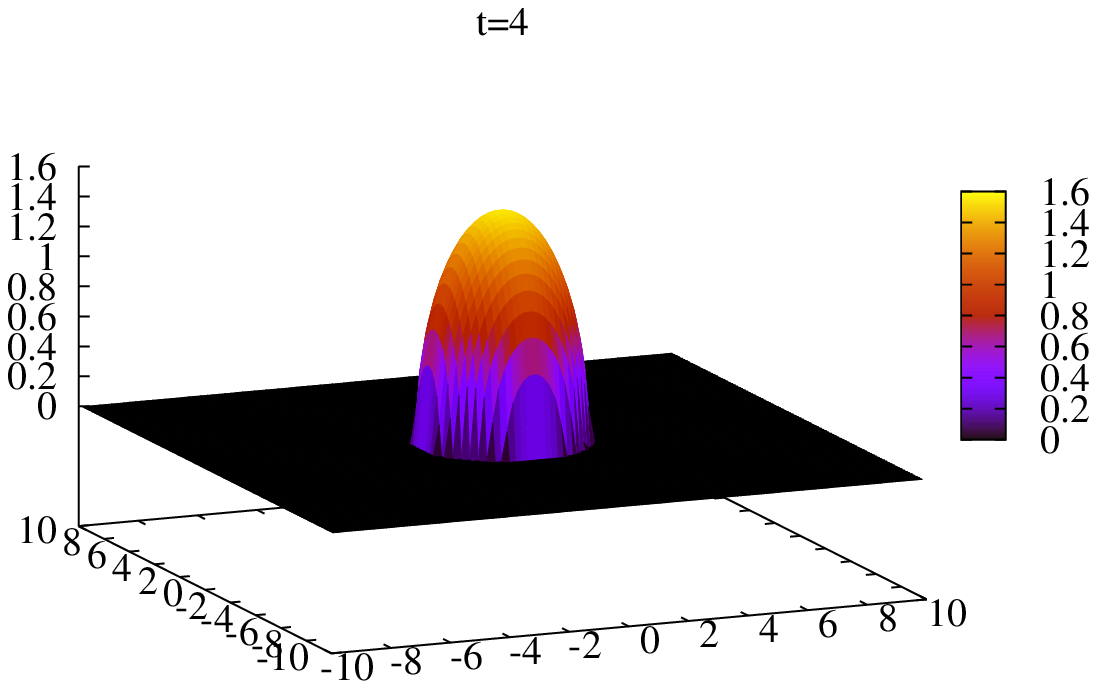}}
\subfigure{\includegraphics[scale=0.55]{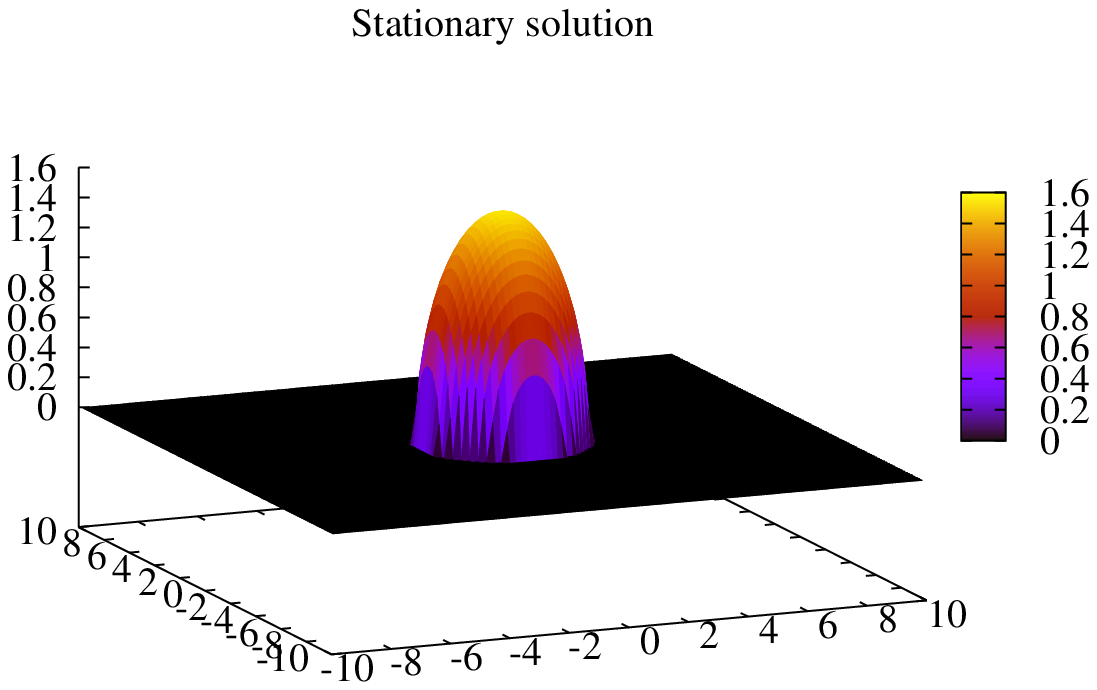}}
\caption{Evolution of the density of the gas $u$ and stationary solution $u^{eq}$.}
\label{evolu}
\end{figure}

\begin{figure}[!ht]
\centering
\subfigure{\includegraphics[scale=0.45]{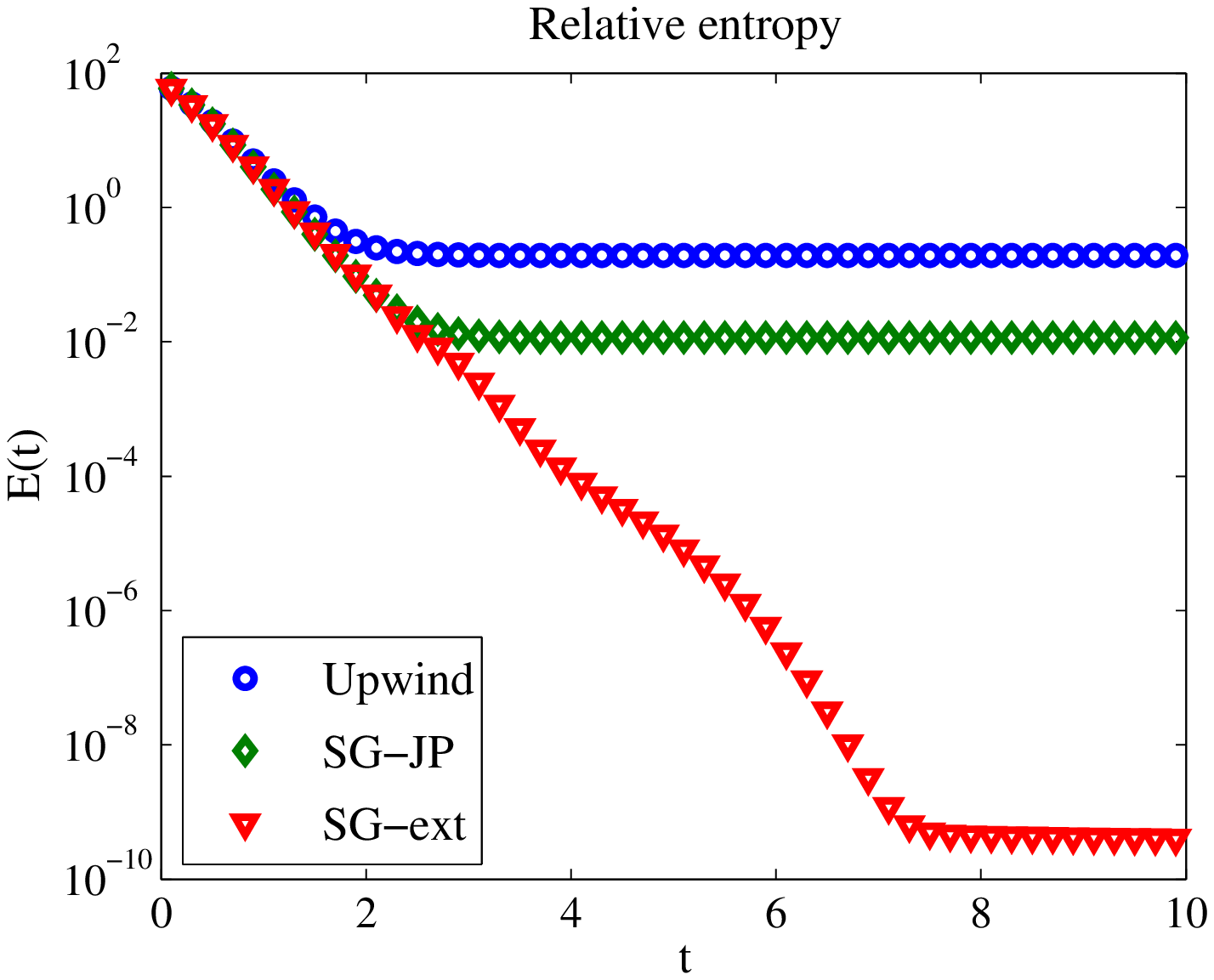}}
\subfigure{\includegraphics[scale=0.45]{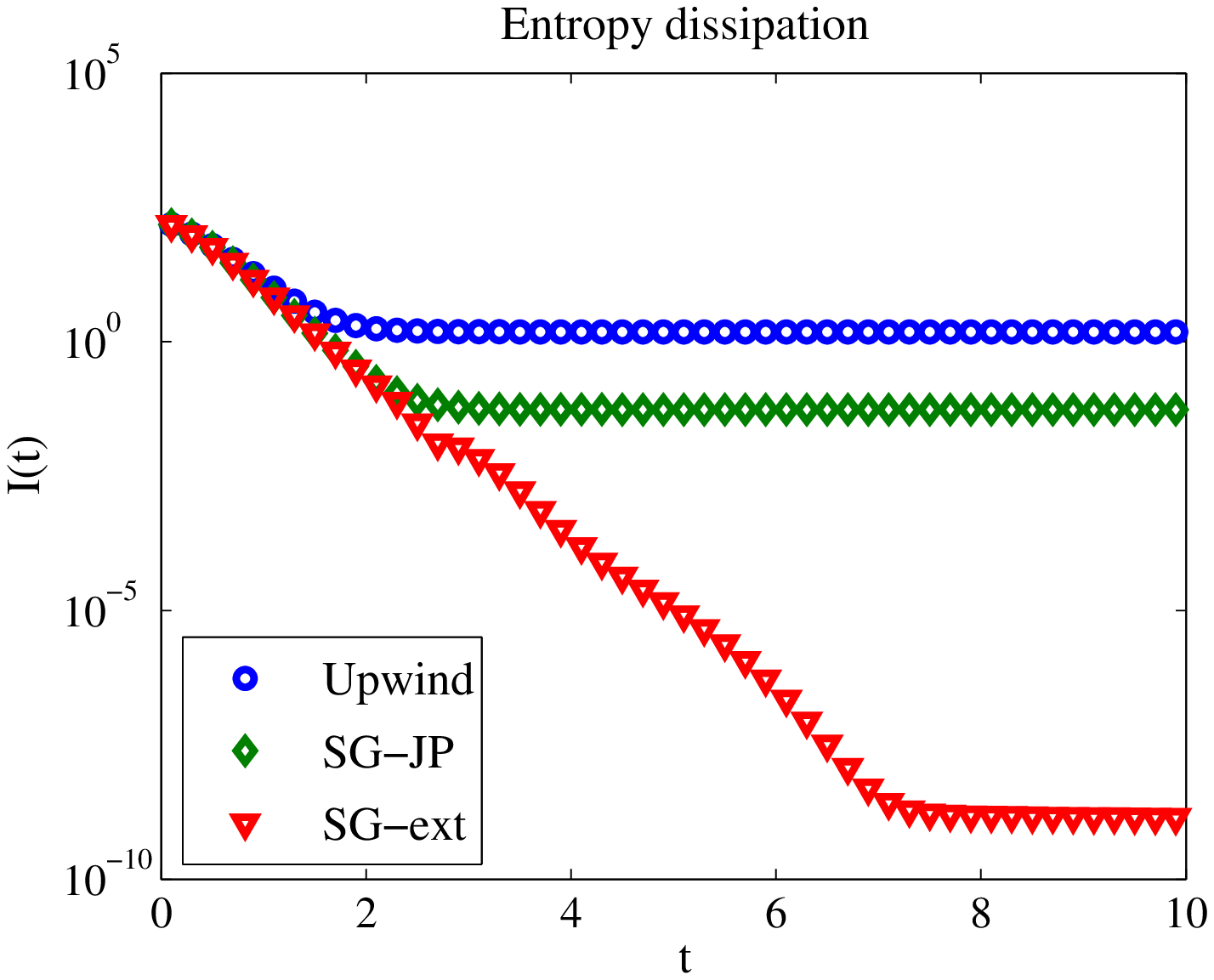}}
\caption{Evolution of the relative entropy $\mathcal{E}^{n}$ and its dissipation $\mathcal{I}^{n}$ in log-scale for different schemes.}
\label{EetImp}
\end{figure}

\begin{figure}[!ht]
\centering
\includegraphics[scale=0.45]{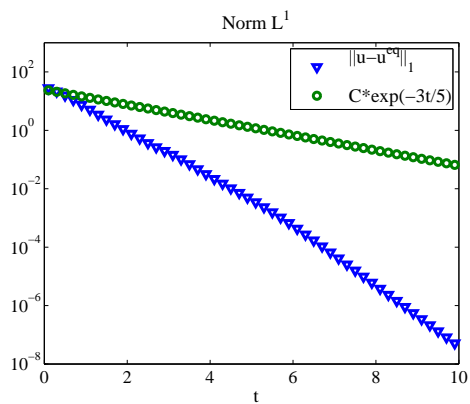}
\caption{Decay rate of $\Vert U-U^{eq} \Vert_{L^{1}}$.}
\label{norml1mp}
\end{figure}


\section{Conclusion}

In this article, we presented how to build a new finite volume scheme for nonlinear convection-diffusion equations. To this end, we have to adapt the Scharfetter-Gummel scheme, in such way that ensures that a particular type of steady-state is preserved. Moreover, this new scheme is easier to implement than existing schemes preserving steady-state.\\
In addition, we have shown that there is convergence of our scheme in the nondegenerate case. The proof of this convergence is essentially based on a discrete $L^{2}\left(0,T;H^{1}\right)$ estimate (\ref{estiml2}). A first step to then prove the convergence in the degenerate case would be to show this estimate without using the uniform lower bound of $u_{\delta}$.\\
Finally, we have observed that this scheme appears to be more accurate than the upwind one, even in the degenerate case. Indeed, we have applied it to the drift-diffusion model for semiconductors as well as to the porous media equation. In these two specific cases, we clearly underlined the efficiency of our scheme in order to preserve long-time behavior of the solutions. At this point, it still remains to prove rigorously this asymptotic behavior, by showing a similar estimate to the one of the continuous framework (\ref{H}) for discrete energy and discrete dissipation.\\
\newline
\textbf{Acknowledgement:} The author is partially supported by the European Research Council ERC Starting Grant 2009, project 239983-NuSiKiMo, and would like to thank C. Chainais-Hillairet and F. Filbet for fruitful suggestions and comments on this work.


\bibliographystyle{plain}
\bibliography{biblio}

\end{document}